\documentclass[12pt, reqno]{amsart}

\usepackage{amsrefs}
\usepackage{amssymb}
\usepackage{amsfonts}
\usepackage{amsmath}
\usepackage{mathrsfs}
\usepackage{comment}
\usepackage{hyperref}

\numberwithin{equation}{section}

\newtheorem{theorem}{Theorem}[section]
\newtheorem{lemma}[theorem]{Lemma}
\newtheorem{proposition}[theorem]{Proposition}

\newtheorem{corollary}[theorem]{Corollary}

\newcommand{\abs}[1]{|#1|}
\newcommand{\C}{\ensuremath{\mathbb{C}^n}}
\newcommand{\Ctwo}{\ensuremath{\mathbb{C}^2}}

\newcommand{\R}{\ensuremath{\mathbb{R}}}
\newcommand{\Z}{\ensuremath{\mathbb{Z}}}
\newcommand{\N}{\ensuremath{\mathbb{N}}}
\renewcommand{\S}{\ensuremath{\text{Sig}_d}}
\newcommand{\Rd}{\ensuremath{{\mathbb{R}^d}}}
\newcommand{\D}{\ensuremath{\mathscr{D}}}

\newcommand{\J}{\ensuremath{\mathscr{J}}}
\newcommand{\F}{\ensuremath{\mathscr{F}}}
\newcommand{\G}{\ensuremath{\mathscr{G}}}
\newcommand{\h}{\ensuremath{\mathscr{H}}}
\newcommand{\BMOd}{\ensuremath{\text{BMO}^d}}
\newcommand{\BMOW}{\ensuremath{{\text{BMO}}_W ^p}}
\newcommand{\BMOWD}{\ensuremath{{\text{BMO}}_{W, \D} ^p}}
\newcommand{\BMOWDt}{\ensuremath{{\text{BMO}}_{W, \D^t} ^p}}

\newcommand{\BMOWq}{\ensuremath{{\text{BMO}}_{W^{1 - p'}} ^{p'}}}

\newcommand{\Mn}{\ensuremath{\mathcal{M}_{n }}(\mathbb{C})}

\newcommand{\inrd}{\ensuremath{\int_{\Rd}}}
\newcommand{\tr}{\ensuremath{\text{tr}}}
\newcommand{\ip}[2]{\ensuremath{\left\langle#1,#2\right\rangle}}
\newcommand{\BMO}{\ensuremath{\text{BMO}}}

\newcommand{\Atwo}[1]{\ensuremath{\|#1 \|_{\text{A}_2} }}
\newcommand{\V}[1]{\ensuremath{\vec{#1}}}
\newcommand{\MC}[1]{\ensuremath{\mathcal{#1}}}
\newcommand{\chG}{\ensuremath{\text{ch}_{\MC{G}}}}

\title[Weighted  inequalities, commutators, and paraproducts]{Matrix weighted norm inequalities for commutators and paraproducts with matrix symbols}

\author[Joshua Isralowitz]{Joshua Isralowitz}
\address[Joshua Isralowitz]{Department of Mathematics and Statistics \\
SUNY Albany \\
1400 Washington Ave. \\
 Albany, NY  \\
12222}
\email[Joshua Isralowitz]{jisralowitz@albany.edu}

\author[Hyun Kyoung Kwon]{Hyun Kyoung Kwon}
\address[Hyun Kyoung Kwon]{Department of Mathematics \\
The University of Alabama \\
P. O. Box 870350 \\
Tuscaloosa, AL \\
35487 }
\email[Hyun Kyoung Kwon]{hkwon@ua.edu
}

\author[Sandra Pott]{Sandra Pott}
\address[Sandra Pott]{Centre for Mathematical Sciences \\ Lund University \\ P.O. Box 118 \\ S-221 00 Lund \\ Sweden}
\email[Sandra Pott]{sandra@maths.lth.se}


\begin{document}

\maketitle

\begin{abstract}
Let $B$ be a locally integrable matrix function, $W$ a matrix A${}_p$ weight with $1 < p < \infty$, and $T$ be any of the Riesz transforms.  We will characterize the boundedness of the commutator $[T, B]$ on $L^p(W)$ in terms of the membership of $B$ in a natural matrix weighted BMO space.  To do this, we will characterize the boundedness of dyadic paraproducts on $L^p(W)$ via a new matrix weighted Carleson embedding theorem. Finally, we will use some of the ideas from these proofs to (among other things) obtain quantitative weighted norm inequalities for these operators and also use them to prove sharp $L^2$ bounds for the Christ/Goldberg matrix weighted maximal function associated with matrix A${}_2$ weights.
\end{abstract}



\section{Introduction}
\label{Intro}
\subsection{Motivation} The $L^p$ boundedness of commutators between functions and Calder\'{o}n-Zygmund operators (or CZOs for short) is a classical subject that was first studied in \cite{CRW} and has numerous applications to PDEs, operator theory, and complex analysis (see \cite{CRW, T} for a small sampling of these applications). Although numerous authors have subsequently used or extended the boundedness results in \cite{CRW}, and although weighted norm inequalities for CZOs have been extensively studied for the past $40$ years or so (starting with the seminal work \cite{HMW}), less attention has been paid towards weighted norm inequalities for commutators.  It is well known, however, that the commutator $[T, b]$ is bounded on $L^p(w)$ (where $T$ is a CZO and $w$ is an A${}_p$ weight) if $b$ is in the classical John-Nirenberg BMO space. Furthermore, it is well known that the boundedness of $[T, b]$ on $L^p(w)$ implies that $b \in \BMO$ when $T$ is one of the Riesz transforms (see \cite{Bl, ABKP} for example.  Also see the interesting preprints \cite{HLW1, HLW2} for a  modern discussion and extensions of the results in \cite{Bl}).

 On the other hand, it is well known that proving matrix weighted norm inequalities for even CZOs is a very difficult task, and because of this, matrix weighted norm inequalities for certain CZOs have only recently been investigated (see \cite{TV, V} for specific details of these difficulties).  In particular, if $n$ and $d$ are natural numbers and if $W : \Rd \rightarrow \Mn$ is positive definite a. e.  (where as usual $\Mn$ is the algebra of $n \times n$ matrices with complex scalar entries), then define $L^p(W)$ for $1 < p < \infty$ to be the space of measurable functions $\vec{f} : \Rd \rightarrow \C$ where \begin{equation*} \|\vec{f}\|_{L^p(W)} ^p = \inrd |W^\frac{1}{p} (x) \vec{f}(x) |^p \, dx < \infty. \end{equation*} It was proved by F. Nazarov and S. Treil, M. Goldberg, and A. Volberg, respectively in \cite{G, NT, V} that certain CZOs acting componentwise on $\C$ valued functions are bounded on $L^p(W)$ when $1 < p < \infty$ if  $W$ is a matrix A${}_p$ weight, which means that \begin{equation} \label{MatrixApDef} \|W\|_{\text{A}_p} := \sup_{\substack{I \subset \R^d \\ I \text{ is a cube}}} \frac{1}{|I|} \int_I \left( \frac{1}{|I|} \int_I \|W^{\frac{1}{p}} (x) W^{- \frac{1}{p}} (t) \|^{p'} \, dt \right)^\frac{p}{p'} \, dx  < \infty \end{equation} where $p'$ is the conjugate exponent of $p$.

Despite this, virtually nothing has been studied regarding matrix weighted norm inequalities for operators related to CZOs that themselves have matrix kernels (in the case of CZOs) or matrix symbols (in the case of operators such as commutators, paraproducts, or Haar multipliers).  The purpose of this paper is to initiate such a study, and in particular, we will characterize matrix weighted norm inequalities for commutators $[T, B]$ when $W$ is a matrix A${}_p$ weight, $B$ is a locally integrable matrix function, and $T$ is any of the Riesz transforms (see also the first author's preprint \cite{I2} where the matrix weighted boundedness of certain matrix kernelled CZOs is investigated).

\subsection{Reducing operators}  We will need to briefly discuss  a very important reformulation of the matrix A${}_p$ condition before we state our main results.  Given any norm $\rho$ on $\C$, the classical ``John's ellipsoid theorem" (see \cite{Bow}) says that there exists a reducing operator $V$ (i.e. a positive definite $n \times n$ matrix $V$) where \begin{equation*}  \rho(\V{e}) \leq      |V \V{e}| \leq \sqrt{n} \rho(\V{e})  \end{equation*}for all $\V{e} \in \C$.    Given a matrix weight $W$, a cube $I$, and some  $1 < p < \infty$, let $V_I = V_I(W, p)$ and $ V_I' = V_I'(W, p) $ be  reducing operators corresponding to the norms \begin{equation*} \rho_{W, I, p} (\V{e}) := \left(\frac{1}{|I|} \int_I |W^\frac{1}{p} (x) \V{e}|^p \, dx \right)^\frac{1}{p} \ \text{ and }  \  \rho_{W, I, p} ^* (\V{e}) := \left(\frac{1}{|I|}  \int_I |W^{-\frac{1}{p}} (x) \V{e}|^{p'} \, dx \right)^\frac{1}{p'} \end{equation*} \noindent on $\C$, respectively. \noindent Note that while these reducing operators are not necessarily unique, the precise reducing operator being used will not be important.  It will be important later, however,  to notice that by definition we can take  $V_I(W^{1 - p'}, p') = V_I ' (W, p)$ and $V_I ' (W^{1 - p'}, p') =  V_I(W, p)$.

  Using these reducing operators and the equivalence of the canonical matrix norm and trace norm on $\Mn$, we have for a matrix A${}_p$ weight $W$ that
\begin{equation*} \|W\|_{\text{A}_p}  = \sup_{\substack{I \subset \R^d \\ I \text{ is a cube}}} \frac{1}{|I|} \int_I \left( \frac{1}{|I|} \int_I \|W^{\frac{1}{p}} (x) W^{- \frac{1}{p}} (t) \|^{p'} \, dt \right)^\frac{p}{p'} \, dx   \approx \sup_{\substack{I \subset \R^d \\ I \text{ is a cube}}} \|V_I V_I' \| ^p \end{equation*} \noindent which also immediately gives us that $W$ is a matrix A${}_p$ weight if and only if $W^{1-p'}$ is a matrix A${}_{p'}$ weight. Furthermore, it is not difficult to see  (using H\"{o}lder's inequality and some elementary arguments involving norms and dual norms, see \cite{G} p. 4) that \begin{equation} |V_I ' V_I \V{e}| \geq |\V{e}| \label{ReverseAp} \end{equation}  for any matrix (not necessarily matrix A${}_p$) weight $W$, any cube $I$, any $1 < p < \infty$, and any $\V{e} \in \C$.

 Also when $p = 2$, a very simple and direct computation shows that we may take  $V_I = (m_I W)^\frac{1}{2}$ and $V_I ' = (m_I (W^{-1}))^{\frac{1}{2}}$ where $m_I  W$ is the average of $W$ on $I$. In particular,  the matrix A${}_2$ condition takes on a particularly simple form that is very similar to the scalar A${}_2$ condition.  Similarly when $W(x) = w(x) \text{Id}_{n \times n}$ for a scalar A${}_p$ weight $w$ we can take $V_I = (m_I w)^\frac{1}{p} \text{Id}_{n \times n}$ and $V_I' = (m_I w^{1-p'})^{\frac{1}{p'}} \text{Id}_{n \times n}$.

 Lastly, it will be useful later in the paper to examine the relationship between $V_I$ and $V_{\tilde{I}}$ where $I, \tilde{I}$ are cubes with $I \subseteq \tilde{I}$ and comparable side-lengths.  In particular, for any $\V{e} \in \C$ we have \begin{equation} |V_I \V{e}|^p  \approx \frac{1}{|I|} \int_I |W^{\frac{1}{p}} (x) \V{e}|^p \, dx \lesssim \frac{1}{|\tilde{I}|} \int_{\tilde{I}} |W^{\frac{1}{p}} (x) \V{e}|^p \, dx \lesssim  |V_{\tilde{I}} \V{e} |^p \label{SmallBig} \end{equation} and a similar computation shows that \begin{equation*} |V_{I} '  \V{e}|^{p'} \lesssim |V_{\tilde{I}} ' \V{e}|^{p'}   \end{equation*} or equivalently \begin{equation*} |(V_{\tilde{I}} ')^{-1} \V{e}|^{p'} \lesssim | (V_{I} ')^{-1}  \V{e}|^{p'}.   \end{equation*} \noindent On the other hand, if $W$ is a matrix A${}_p$ weight then the inequality above combined with the the A${}_p$ condition gives us that \begin{equation} |V_{\tilde{I}} \V{e} |^p  \lesssim \|V_{\tilde{I}} V_{\tilde{I}}' \|^p  |(V_{\tilde{I}}' )^{-1} \V{e} |^p  \leq \|W\|_{\text{A}_p} |(V_{I} ')^{-1} \V{e} |^p \leq \|W\|_{\text{A}_p} |V_I  \V{e} |^p \label{BigSmall} \end{equation} where the last line follows from \eqref{ReverseAp}.

\subsection{Notation and main results}
 Now if $1 < p < \infty$ and $W$ is a matrix A${}_p$ weight, then let $\BMOW$ be the space of locally integrable functions $B : \Rd \rightarrow \Mn$ where  \begin{equation*}     \left\{
     \begin{array}{lr}
     \displaystyle   \sup_{\substack{I \subset \R^d \\  I \text{ is a cube}}} \frac{1}{|I|} \int_I \|W^\frac{1}{p} (x) (B(x) - m_I B) V_I ^{-1}  \|^p \, dx < \infty  & :\text{ if } 2 \leq p  < \infty \\
     \displaystyle   \sup_{\substack{I \subset \R^d \\  I \text{ is a cube}}} \frac{1}{|I|} \int_I \|W^{-\frac{1}{p}} (x) (B^* (x) - m_I B^*) (V_I ')^{-1}  \|^{p'} \, dx < \infty & : \text{ if } 1 < p \leq 2.
     \end{array}
   \right.
\end{equation*}

\noindent Also, given a dyadic grid $\D$,  we will let $\BMOWD$ denote the space of locally integrable $n \times n$ functions satisfying the condition above but where the supremum is taken over all $I \in \D$. Note that these two conditions should be thought of as dual to each other (in a precise sense that will be explained later in this introduction.)   The main result of this paper is the following

\begin{theorem}\label{CommutatorThm}  Let $1 < p < \infty.$  If $W$ is a matrix A${}_p$ weight and $T$ is any of the Riesz transforms, then $[T, B]$ is bounded on $L^p(W)$   if and only if $B \in \BMOW$. \end{theorem}

As is well known, the study of such commutators is often reduced to the study of paraproducts, and this is the approach we will take for proving Theorem \ref{CommutatorThm}. Before we define our paraproducts, let us review some definitions and notation regarding Haar functions in several variables.  Following the notation in \cite{LPPW}, for any dyadic grid in $\mathbb{R}$ and any interval in this grid, let  \begin{equation*} h_I ^1 = |I|^{-\frac{1}{2}} \chi_I (x), \,  \,  \,  \, \, \, \,  h_I ^0 (x) = |I|^{-\frac{1}{2}} (\chi_{I_\ell} (x) - \chi_{I_r} (x)).  \end{equation*} Now given any dyadic grid $\D$ in $\mathbb{R}^d,$  any cube $I = I_1 \times \cdots \times I_d$, and any $\varepsilon \in \{0, 1\}^{d}$, let $h_I ^\varepsilon = \Pi_{i = 1}^d h_{I_i} ^\varepsilon$.  It is then easily seen that $\{h_I ^\varepsilon\}_{I \in \D, \  \varepsilon \in \S}$  where $\S = \{0, 1\}^d \backslash \{\vec{1}\}$ is an orthonormal basis for $L^2(\Rd)$. We will say $h_I ^\varepsilon$ is ``cancellative" if $\varepsilon \neq \vec{1}$ since in this case $\int_I h_I^\varepsilon = 0$.

Now given a locally integrable function $B : \Rd \rightarrow \Mn$, define the dyadic paraproduct $\pi_B$ with respect to a dyadic grid $\D$ by \begin{equation} \label{ParaprodDef} \pi_B \vec{f} = \sum_{\varepsilon \in \S} \sum_{I \in \D} B_I ^\varepsilon (m_I \vec{f}) h_I ^\varepsilon \end{equation}
 where $B_I ^\varepsilon$ is the matrix of Haar coefficients of the entries of $B$ with respect to $I$ and $\varepsilon$,  and $m_I \vec{f}$ is the vector of averages of the entries of $\vec{f}$.  The proof of Theorem \ref{CommutatorThm} will be largely based on the following

 \begin{theorem}  \label{ParaThm} Let $1 < p < \infty$ and let $\D$ be a dyadic grid.  If $W$ is a matrix A${}_p$ weight then the paraproduct $\pi_B$ with respect to $\D$ is bounded on $L^p(W)$ if and only if $B \in \BMOWD$.  \end{theorem}

\noindent Note that the proofs of Theorems \ref{CommutatorThm} and \ref{ParaThm} actually give us quantitative bounds when $p = 2$.  In particular,  we will prove that \begin{equation} \|\pi_B \|_{L^2(W) \rightarrow L^2(W)} \lesssim (\log \Atwo{W}) ^\frac12 \Atwo{W} ^\frac32 \|B\|_* ^\frac12   \label{ParaQuant}\end{equation}  where  $\|B\|_*$ is the canonical supremum in condition (b) of Theorem \ref{CarEmbedThm} below.   Moreover, we will prove that \begin{align}  \|[T, B] \|_{L^2(W) \rightarrow L^2(W)} & \lesssim \|Q\|_{L^2(W) \rightarrow L^2(W)} \max \{\|\pi_B\|_{L^2(W)\rightarrow L^2(W)} ,  \|\pi_{B^*}\|_{L^2(W^{-1} )\rightarrow L^2(W^{-1})} \nonumber \\ &  +  \Atwo{W} ^\frac32 \log \Atwo{W} \|B\|_{*} ^\frac12 \} \label{CommQuant} \end{align} where  $T$ is any of the Riesz transforms and $Q$ is a first order Haar shift (see Section \ref{Section31} for the definition.)     It would be very interesting to know if any similar commutator bounds for general scalar CZOs are true, and this will be explored in a forthcoming paper by the first and third authors.

Besides being extremely important for proving results regarding commutators (see \cite{LPPW, HLW1, HLW2} for example), note that paraproducts are central to the study of CZOs themselves since they allow one to decompose an arbitrary CZO $T$ as $T = \pi_{T1} + \pi_{T^*1} ^*  + R$ where $R$ is cancellative in the sense that $R1 = R^*1 = 0$. In fact, the first author in \cite{I2} will employ Theorem \ref{ParaThm} to prove a  T1 theorem regarding the matrix weighted boundedness of certain matrix kernelled CZOs.

 The proof of Theorem \ref{ParaThm} will easily follow from the following matrix weighted Carleson embedding theorem (see the next section for details), which is obviously of independent interest itself. Here, for a dyadic grid $\D$ and $J \in \D$, we define $\D(J) = \{I \in \D : I \subseteq J\}$.

  \begin{theorem} \label{CarEmbedThm} Let $1 < p < \infty$ and let $\D$ be a dyadic grid.  If  $W$ is a matrix A${}_p$ weight and $ A := \{A_I ^\varepsilon \}_{I \in \D, \varepsilon \in \S}$ is a sequence of matrices, then the following are equivalent

\begin{itemize}
 \item[(a)]  The operator $\Pi_A$ defined by \begin{equation*} \Pi_A \vec{f} := \sum_{\varepsilon \in \S} \sum_{I \in \D} V_I A_I ^\varepsilon  m_I ( W^{- \frac{1}{p}} \vec{f}) h_I ^\varepsilon \end{equation*} is bounded on $L^p(\Rd;\C)$.
\item[(b)]  \begin{equation*} \sup_{J \in \D} \frac{1}{|J|} \sum_{\varepsilon \in \S} \sum_{I \in \D(J)} \|V_I A_I ^\varepsilon V_I   ^{-1}  \|^2 < \infty. \end{equation*}
\item[(c)] There exists $C > 0$ independent of $J \in \D$ such that \begin{equation*}  \frac{1}{|J|}  \sum_{\varepsilon \in \S} \sum_{I \in \D(J)} (A_I ^\varepsilon) ^* V_I^2 A_I ^\varepsilon  < C V_J ^2  \end{equation*} if $2 \leq p < \infty$, and  \begin{equation*}  \frac{1}{|J|}  \sum_{\varepsilon \in \S} \sum_{I \in \D(J)} A_I ^\varepsilon (V_I ')^2 (A_I ^\varepsilon) ^*  < C (V_J ') ^2  \end{equation*} if $1 < p \leq 2$.
\end{itemize}

\noindent Furthermore, the operator norm in $(a)$ and the canonical supremums in $(b)$ and $(c)$ are equivalent in the sense that they are independent of the sequence $A$.  Finally, a matrix function $B \in \BMOWD$ if and only if the sequence of Haar coefficients of $B$ satisfies any of the above equivalent conditions.
 \end{theorem}

\noindent Note that the constants in the equivalence between the operator norm in $(a)$ and the canonical supremums in $(b)$ and $(c)$ of course depend on the A${}_p$ characteristic of $W$, and throughout the proof we will track precisely the nature of this dependence (modulo constants involved in the matrix weighted Triebel-Lizorkin imbedding theorem when $p \neq 2$, since in this case efficient bounds are not known, see Section \ref{SectionTLT} for more details).  Also, note that (as to be expected), we have the following relationship between $\BMOW$ and $\BMOWD$ \begin{proposition}  \label{BMOvsDyadicBMO} There exists dyadic grids $\D^t$ for $t = 1, \ldots, 2^d$ where \begin{equation*} \BMOW  = \bigcup_{t = 1}^{2^d} \BMOWDt. \end{equation*} \end{proposition}

  Note that while the two different cases for different $p$ in the definition of $\BMOW$ might seem awkward, it will turn out that Theorem \ref{CommutatorThm} and duality will together prove the following \begin{corollary} \label{BMODefCor}  If $1 < p < \infty$ and $W$ is a matrix A${}_p$ weight, then $B \in \BMOW$ if and only if both  \begin{equation}    \sup_{\substack{I \subset \R^d \\  I \text{ is a cube}}} \frac{1}{|I|} \int_I \|W^\frac{1}{p} (x) (B(x) - m_I B) V_I ^{-1}  \|^p \, dx < \infty \label{Cond} \end{equation} and the dual condition\begin{equation} \sup_{\substack{I \subset \R^d \\  I \text{ is a cube}}} \frac{1}{|I|} \int_I \|W^{-\frac{1}{p}} (x) (B^* (x) - m_I B^*) (V_I ')^{-1}  \|^{p'} \, dx < \infty \label{DualCond} \end{equation} are true. \end{corollary} \noindent

Finally, to prove Theorem \ref{CommutatorThm} we will need to characterize matrix weighted norm inequalities for Haar multipliers.  More precisely we will prove the following

\begin{proposition} \label{HaarMultThm} Let $1 < p < \infty$ and let $W$ be a matrix A${}_p$ weight.  If $\D$ is any dyadic grid and $A := \{A_I ^\varepsilon \}_{I \in \D, \varepsilon \in \S}$ is a sequence of matrices, then the Haar multiplier \begin{equation*} T_A \vec{f} := \sum_{I \in \D} \sum_{\varepsilon \in \S} A_I ^\varepsilon {\vec{f}}_I ^\varepsilon  h_I ^\varepsilon \end{equation*} is bounded on $L^p(W)$ if and only if \begin{equation*} \sup_{I \in \D, \varepsilon \in \S} \|V_I A_I ^\varepsilon V_I ^{-1}\| < \infty. \end{equation*}  \end{proposition}

\subsection{The scalar setting}
Let us now make a few comments about these results in the scalar seting.  First,  it is very easy to see that a scalar function $b$ is in $\BMOW$ if and only if $b \in \text{BMO}$ when $W$ is a matrix A${}_p$ weight of the form $W(x) = w(x) \text{Id}_{n \times n}$ for a scalar A${}_p$ weight $w$.  In particular, for each dyadic grid $\D$, condition (a) in Theorem \ref{CarEmbedThm} is trivially equivalent to $b \in \text{BMO}_{\D}$, which given Proposition \ref{BMOvsDyadicBMO} clearly proves the claim. Furthermore, it is well known that $\pi_b$ is bounded on $L^p(w)$ if and only if $b \in \BMO$ (when $w$ is a scalar A${}_p$ weight, see \cite{Bez}).

Moreover, note that when $p = 2$,  a careful tracking of the $\Atwo{W}$ characteristic contribution from the implication (c) $\Rightarrow$ (a) in Theorem \ref{CarEmbedThm} gives us the following  (after replacing $\vec{f}$ with $W^\frac{1}{2} \vec{f}$ and replacing $A_I ^\varepsilon$ with $(m_I W) ^{- \frac{1}{2}} A_I ^\varepsilon$ )

\begin{corollary} \label{CarEmbedCondp=2} If $W$ is a matrix A${}_2$ weight, $\D$ is a dyadic grid, and $\{A_I ^\varepsilon\}$ is any sequence of $n \times n$ matrices satisfying

\begin{equation} \sum_{\varepsilon \in \S}\sum_{I \in \D(J)}  (A_I ^\varepsilon) ^* A_I ^\varepsilon     < C \int_J W (x) \, dx  \nonumber  \end{equation}  for all $J \in \D$ (where $C$ is independent of $J$) then
\begin{equation} \sum_{\varepsilon \in \S} \sum_{I \in \D} |A_I ^\varepsilon (m_I \vec{f})|^2 \lesssim C ^\frac12 \Atwo{W}^3 \|\vec{f}\|_{L^2(W)} ^2. \nonumber \end{equation} \end{corollary}

Interestingly, note that Corollary \ref{CarEmbedCondp=2} in the scalar $d = 1$ setting appears as Lemma $5.7$ in \cite{P} for scalar A${}_\infty$ weights and was implicitly used in sharp form with quadratic $\Atwo{W}$ characteristic (versus cubic above) by O. Beznosova in \cite{Bez} (see $(2.4)$ and $(2.5)$ in \cite{Bez}) to prove sharp weighted norm inequalities for scalar paraproducts.

Also, a similar $p = 2$ matrix weighted Carleson embedding theorem for positive semidefinite  sequences was proved in \cite{BW1} using virtually the same argument as the one used to prove Theorem \ref{CarEmbedThm}. Additionally, note that a version of Corollary \ref{CarEmbedCondp=2} for positive semidefinite sequences that does not require $W$ to be a matrix A${}_2$ weight was very recently proved in \cite{CT}.  While this result is obviously of great potential for proving sharp matrix A${}_2$ results, it is not clear whether one can prove Corollary \ref{CarEmbedCondp=2} using the results in \cite{CT}.

\subsection{Outline of paper}
We will now briefly outline the contents of the paper.  In Section $2$ we will prove Theorems \ref{CarEmbedThm},  \ref{ParaThm}, and also prove  \eqref{ParaQuant}.  In Section $3$ we will prove Proposition \ref{HaarMultThm} and use this in conjunction with Theorem \ref{ParaThm} to prove Theorem \ref{CommutatorThm} and \eqref{CommQuant}.  Additionally, we will give short proofs of Proposition \ref{BMOvsDyadicBMO} and Corollary \ref{BMODefCor} in Section $3$.  Finally,   in the last section we will provide very explicit ``counterexamples" to Proposition \ref{HaarMultThm}, Theorem \ref{ParaThm}, and Theorem \ref{CommutatorThm} in the sense that, as one would expect, none of these results are true for arbitrary matrix valued symbols and matrix A${}_2$ weights.

We will also prove some simple yet nonetheless interesting results involving quantitative matrix weighted norm inequalities for objects related to maximal functions.  In particular, we will prove sharp $L^2$ estimates for the Christ/Goldberg matrix weighted maximal function from  \cite{CG}, prove weak type estimates for ``the" universal $p = 2$ matrix weighted maximal function for not necessarily matrix A${}_2$ weights, and give a simple ``maximal function" proof of the matrix weighted bounds for sparse operators from \cite{BW}.   While these results are not needed to prove any of our main results, their proofs require some ideas utilized in this paper and clearly complement \eqref{ParaQuant} and \eqref{CommQuant}

Finally, we will remark that well after this paper was written, the first author in \cite{I3} has proved that in fact $B \in \BMOW \Leftrightarrow \eqref{Cond} \Leftrightarrow \eqref{DualCond}$, and that a similar result holds for $\BMOWD$.  Furthermore, it is very interesting to note that most of the techniques in this paper are in fact ``two weight" techniques in that slight modifications to them allow for extensions of Theorems \ref{ParaThm} and \ref{CommutatorThm} to the $L^p(U) \rightarrow L^p(W)$ setting where $U, W$ are matrix A${}_p$ weights, see \cite{I3} for more details.

\section{Proof of Theorem $\ref{ParaThm}$} The main goal of this section is to prove Theorem \ref{ParaThm} via Theorem \ref{CarEmbedThm}.  Before we do either, however, we will need to discuss some preliminary results.

\subsection{Matrix weighted Littlewood-Paley theory} \label{SectionTLT}

We will now need the ``matrix weighted Triebel-Lizorkin imbedding theorem" from \cite{NT, V}, which say that if $W$ is a matrix A${}_p$ weight then \begin{equation} \label{LpEmbedding}  \|\vec{f}\|_{L^p(W)} ^p  \approx \inrd \left(\sum_{I \in \D} \sum_{\varepsilon \in \S}  \frac{| V_I \vec{f}_I ^\varepsilon|^2}{|I|} \chi_I(x) \right)^\frac{p}{2} \, dx  \end{equation} where $\vec{f}_I ^\varepsilon$ is the vector of Haar coefficients of the components of $\vec{f}$. Note that these were only proved in the $d = 1$ setting in \cite{NT, V}, though a very simple proof that works for $\Rd$ was given by the first author in \cite{I1}. Furthermore, note that when $p = 2$ the above ``Littlewood-Paley expression" reduces to a matrix weighted dyadic square function, and in this setting it is known that one has the quantitative bounds (see \cite{BPW} for $d = 1$ and \cite{CW} for $d > 1$)

\begin{equation} \label{UpperBound} \left( \sum_{I \in \D} \sum_{\varepsilon \in \S} |(m_I W )^\frac12 \vec{f}_I ^\varepsilon |^2 \right)^\frac12 \lesssim \Atwo{W} (\log \Atwo{W})^\frac12 \|\vec{f}\|_{L^2(W)} \end{equation} and \begin{equation} \label{LowerBound} \|\vec{f}\|_{L^2(W)} \lesssim  \Atwo{W} ^\frac12  (\log \Atwo{W})^\frac12 \left( \sum_{I \in \D} \sum_{\varepsilon \in \S} |(m_I W )^\frac12 \vec{f}_I ^\varepsilon |^2 \right)^\frac12.  \end{equation} Unfortunately, while one can attempt to track the matrix A${}_p$ dependence in \eqref{LpEmbedding} from the arguments in either \cite{NT}, \cite{TV}, or \cite{I1}, when $p \neq 2$, it is very unlikely that any of the arguments in these papers provide efficient bounds similar to the ones in \eqref{UpperBound} or \eqref{LowerBound} when $p = 2$. With this in mind, it will be implicit that all inequalities involving \eqref{LpEmbedding} when $p \neq 2$ involve matrix A${}_p$ dependence and we will not further comment on this.

\subsection{Preliminary lemmas}

Before we prove Theorem \ref{CarEmbedThm} we will need the following three preliminary results, the first of which is from \cite{NT}, p. 49.

\begin{lemma} \label{MatrixNormLem} Suppose that $A$ is an $n \times n$ matrix where $|A \vec{e}| \geq |\vec{e}|$ for any $\vec{e} \in \C$.  If $|\det A | \leq \delta $ for some $\delta \geq 0$, then $\|A\| \leq \delta$ where $\|\cdot\|$ is the canonical matrix norm on $\Mn$.    \end{lemma}

\begin{lemma}\label{RedOp-AveLem}  If $W$ is a matrix A${}_p$ weight then \begin{equation*} |  V_I '  \vec{e} | \approx  |m_I (W^{-\frac{1}{p}}) \vec{e}  | \end{equation*} for any $\vec{e} \in \C$.   In particular, \begin{equation*} |m_I (W^{-\frac{1}{p}}) \vec{e}| \leq |V_I ' \vec{e} | \leq \|W\|_{\text{A}_p} ^\frac{n}{p}  |m_I (W^{-\frac{1}{p}}) \vec{e}|.  \end{equation*}  \end{lemma}

\begin{proof} First we show that \begin{equation*} \|V_I' \left(m_I (W^{-\frac{1}{p}}) \right)^{-1}\| \leq \|W\|_{\text{A}_p} ^\frac{n}{p},  \end{equation*}  which will prove half of the lemma.  Furthermore, note that the proof of this inequality will in fact also complete the other half of the proof.  Since $W$ is a matrix A${}_{p}$ weight, Jensen's inequality gives us that \begin{equation} \label{JenEq} \exp\left[ \frac{1}{|I|} \int_I \log |W^{\frac{1}{p}} (x) \vec{e}| \, dx \right] \leq \|W\|_{\text{A}_p} ^\frac{1}{p}  |(V_I ')^{-1} \vec{e}|. \end{equation}

We now prove that $\det V_I' (m_I (W^{- \frac{1}{p}})^{-1}) \leq \|W\|_{\text{A}_p} ^\frac{n}{p}$ by using some arguments in the proof of Proposition $2.2$ in \cite{V}.  First, as was commented in \cite{V}, $\det Q \leq \Pi_{i = 1}^n |Q \vec{e}_i|$ for any orthonormal basis $\{\vec{e}_i\}_{i = 1}^n$ of $\C$ and any positive definite $Q$.  Now for fixed $I$ let $\{\vec{e}_i\}_{i = 1} ^n$ be an orthornormal basis of $\C$ consisting of eigenvectors of $(V_I ')^{-1}$.  Applying \eqref{JenEq} to each $\vec{e}_i$, taking logarithms, summing,  and using the above inequality gives us that \begin{align*} \frac{1}{|I|} \int_I \log \det W^{\frac{1}{p}} (x) \, dx  & \leq
\sum_{i = 1}^n \frac{1}{|I|} \int_I  \log |W^{\frac{1}{p}}(x) \vec{e}_i | \, dx \\ & \leq \log \|W\|_{\text{A}_p} ^{\frac{n}{p}} + \log \Pi_{i = 1}^n  |(V_I ') ^{-1} \vec{e}_i| \\ & = \log \|W\|_{\text{A}_p} ^{\frac{n}{p}} + \log \det (V_I ')^{-1} \end{align*}
 or equivalently \begin{equation*} \log \det V_I ' \leq \log \|W\|_{\text{A}_p} ^\frac{n}{p} + \frac{1}{|I|} \int_I  \log \det W^{-\frac{1}{p}} (x) \, dx \end{equation*} so that \begin{equation*} \det V_I ' \leq \|W\|_{\text{A}_p} ^\frac{n}{p} \exp \left( m_I \log \det (W^{-\frac{1}{p}})  \right) \end{equation*} for any $I \in \D$.  Combining this with the matrix Jensen inequality (Lemma $7.2$ in \cite{NT}) we have that \begin{equation*}  \det V_I ' \leq \|W\|_{\text{A}_p} ^\frac{n}{p} \det m_I (W^{-\frac{1}{p}}) \end{equation*} so that $\det V_I' (m_I (W^{- \frac{1}{p}})^{-1}) \leq \|W\|_{\text{A}_p} ^\frac{n}{p}$.

Moreover, note that for any $\vec{e} \in \C$ we have \begin{align*} |m_I (W^{- \frac{1}{p}}) \vec{e} |  \leq \frac{1}{|I|} \int_I |W^{- \frac{1}{p}} (x) \vec{e} | \, dx   \leq \left(\frac{1}{|I|} \int_I |W^{- \frac{1}{p}} (x) \vec{e} |^{p'} \, dx \right)^\frac{1}{p'} \leq |V_I ' \vec{e}| \end{align*} which means that \begin{equation*} |V_I' (m_I (W^{- \frac{1}{p}})^{-1}) \vec{e} | \geq |\vec{e}| \end{equation*} for any $\vec{e} \in \C$.   The proof now follows immediately from Lemma \ref{MatrixNormLem}.

\end{proof}

Finally, we will need the following $\Rd$ version of the classical dyadic Carleson Lemma from \cite{P}.  Note that the proof is almost identical to the proof of Lemma $5.3$ in \cite{P} and will therefore be omitted.

\begin{proposition} \label{CarLem} Let $\{\lambda_I ^\varepsilon : I \in \D, \, \varepsilon \in \S\}$ be a ``Carleson sequence" of positive numbers in the sense that \begin{equation*} \sup_{J \in \D} \frac{1}{|J|} \sum_{I \in \D(J)} \sum_{\epsilon \in \S} \lambda_I ^\varepsilon \leq C < \infty. \end{equation*}  Then for any positive sequence of real numbers $\{a_I\}$, we have that \begin{equation*}\sum_{I \in \D} \sum_{\varepsilon \in \S} a_I \lambda_I ^\varepsilon \leq C \inrd a^*(x) \, dx \end{equation*} where $a^*(x) = \sup_{I \ni x} a_I $.  \end{proposition}

\subsection{Stopping time} Let us now review the surprisingly useful stopping time from \cite{I1}, which is a matrix $p \neq 2$ adaption of the stopping time from \cite{KP, Pott}.    Assume that $W$ is a matrix A${}_p$ weight. For any cube $I \in \D$, let $\J(I)$ be the collection of maximal $J \in \D(I)$ such that \begin{equation}      \|V_J V_I^{-1}\|^p  > \lambda_1 \    \text{   or   }  \  \|V_J^{-1} V_I\|^{p'}  > \lambda_2 \label{STDef} \end{equation} for some $\lambda_1, \lambda_2 > 1$ to be specified later.  Also, let $\F(I)$ be the collection of dyadic subcubes of $I$ not contained in any cube $J \in \J(I)$, so that clearly $J \in \F(J)$ for any $J \in \D(I)$.

Let $\J^0 (I) := \{I\}$ and inductively define $\J^j(I)$ and $\F^j(I)$ for $j \geq 1$ by $\J^j (I) := \bigcup_{J \in \J^{j - 1} (I)} \J(J)$ and $\F^j (I) := \bigcup_{J \in \J^{j - 1} (I)} \F(J)$. Clearly the cubes in $\J^j(I)$ for $j > 0$ are pairwise disjoint.  Furthermore, since $J \in \F(J)$ for any $J \in \D(I)$, we have that $\D(I) = \bigcup_{j = 0}^\infty \F^j(I)$.  We will slightly abuse notation and write $\bigcup \J(I)$ for the set $\bigcup_{J \in \J(I)} J$ and write $|\bigcup \J(I)|$ for $|\bigcup_{J \in \J(I)} J|$.  We will now show that $\J$ is a decaying stopping time in the sense of \cite{KP}. Note that while the easy proof is from \cite{I1}, we will include the details since we will need to precisely track the A${}_p$ characteristic contribution.

\begin{lemma} \label{DSTLem}
Let $1 < p < \infty$ and let $W$ be a matrix A${}_p$ weight.  For $\lambda_1, \lambda_2 > 1$ large enough, we have that $|\bigcup \J ^j (I)| \leq 2^{-j} |I|$ for every $I \in \D$.  \end{lemma}


\begin{proof} By iteration, it is enough to prove the lemma for $j = 1$.  For $I \in \D$, let $\G(I)$ denote the collection of maximal $J \in \D(I)$ such that the first inequality (but not necessarily the second inequality) in $(\ref{STDef})$ holds.  Then by maximality and elementary arguments involving the definition of $V_I$ and the equivalence between the matrix and trace norm for positive matrices, we have that \begin{equation*}  \left| \bigcup_{J \in \G(I)} J \right| = \sum_{J \in \G(I)} |J| \lesssim \frac{1}{\lambda_1} \sum_{J \in \G(I)} \int_{J} \|W^\frac{1}{p} (y) V_I ^{-1}\|^p \, dy \leq  \frac{C_1 |I|}{\lambda_1} \end{equation*} for some $C_1 > 0$ only depending on $n$ and $d$.

On the other hand, let  $I \in \D$, let $\widetilde{\G}(I)$ denote the collection of maximal $J \in \D(I)$ such that the second inequality (but not necessarily the first inequality) in $(\ref{STDef})$ holds. Then by the matrix A${}_p$ condition we have  \begin{equation*}  \left| \bigcup_{J \in \widetilde{\G}(I)} J \right|   \leq  \frac{C_2}{\lambda_2}  \sum_{J \in \widetilde{\G}(I)} \int_J \|W^{-\frac{1}{p}} (y) V_I \|^{p'} \, dy  \leq  \frac{C_2 '\|W\|_{A_p} ^\frac{p'}{p}}{\lambda_2} |I| \end{equation*} for some $C_2 '$ only depending on $n$ and $d$.  The proof is now completed by setting $\lambda_1 = 4C_1$ and $\lambda_2 = 4 C_2'  \|W\|_{A_p} ^\frac{p'}{p}$.   \end{proof}

While we will not have a need to discuss matrix A${}_{p, \infty}$ weights in detail in this paper, note that in fact Lemma $3.1$ in \cite{V} immediately gives us that Lemma \ref{DSTLem} holds for matrix A${}_{p, \infty}$ weights (with a different $\lambda_2$ of course.)

\subsection{Proofs}

We now prove Theorem \ref{CarEmbedThm} \\

\noindent \textit{Proof of Theorem} \ref{CarEmbedThm}. (b) $\Rightarrow$ (a): By dyadic Littlewood-Paley theory, we need to show that  \begin{align} \int_{\Rd} & \left( \sum_{\varepsilon \in \S} \sum_{I \in \D} \frac{ |V_I A_I ^\varepsilon m_I(W^{-\frac{1}{p}} \vec{f} )| ^2 }{|I|} \chi_I(t) \right)^\frac{p}{2} \, dt
\nonumber \\ & \leq \int_{\Rd}  \left( \sum_{\varepsilon \in \S} \sum_{I \in \D} \frac{ (\|V_I A_I ^\varepsilon V_I ^{-1} \| m_I |V_I W^{-\frac{1}{p}} \vec{f}|)^2 }{|I|} \chi_I(t) \right)^\frac{p}{2} \, dt \label{CarEmbedParEst}\\ & \lesssim\|\vec{f}\|_{L^p} ^p \nonumber  \end{align} for any $\vec{f} \in L^p(\Rd;\C)$.

Now let \begin{equation*} \tilde{A} = \sum_{\varepsilon \in \S} \sum_{I \in \D} \|V_I A_I ^\varepsilon V_I ^{-1} \| h_I ^\varepsilon \end{equation*} and let \begin{equation*} M_W ' \vec{f} (x) = \sup_{\D \ni I \ni x}  m_I |V_I W^{-\frac{1}{p}} \vec{f}| \end{equation*} Clearly for any $\D \ni I \ni x$ we have that \begin{equation*} m_I |V_I W^{-\frac{1}{p}} \vec{f}| \leq m_I M_W ' \vec{f} \end{equation*} so that again by dyadic Littlewood-Paley theory \begin{align*}  \eqref{CarEmbedParEst} & \leq \int_{\Rd}  \left( \sum_{\varepsilon \in \S} \sum_{I \in \D} \frac{ (\|V_I A_I ^\varepsilon V_I ^{-1} \| m_I( M_W ' \vec{f}))^2 }{|I|} \chi_I(t) \right)^\frac{p}{2} \, dt
\\ & \lesssim \|\pi_{\tilde{A}} M_W ' \vec{f}\|_{L^p} ^p \\ &  \lesssim \|A\|_* ^p \|M_W ' \vec{f}\|_{L^p} ^p  \end{align*}  where $ \|A\|_*$ is the canonical supremum from condition (b) and $\pi_{\tilde{A}}$ is the scalar dyadic paraproduct with respect to the function $\tilde{A}$.

However, it is easy to see that \begin{equation*} \|M_W ' \|_{L^p \rightarrow ^p} \lesssim \|W\|_{\text{A}_p} ^{\frac{1}{p - 1}} \end{equation*} by using some simple ideas from \cite{G}.  Namely, it is well known (see \cite{G}, p. 6) that $|W^{-\frac{1}{p}} (x) \V{e}|^{p'}$ is a scalar A${}_{p'}$ weight for any matrix A${}_p$ weight $W$ and $\V{e} \in \C$ with A${}_{p'}$ characteristic  less than or equal to $\|W\|_{\text{A}_p} ^{-\frac{p'}{p}} $.  Thus, by the scalar reverse H\"{o}lder inequality, the matrix A${}_p$ condition, and the equivalence of the operator and trace norms, we have for $\epsilon  \approx \|W\|_{\text{A}_p} ^{-\frac{p'}{p}} $ small enough that \begin{align*} \left(\frac{1}{|I|}  \int_I  \|V_I W^{-\frac{1}{p}} (y) \|^{p' + p' \epsilon}  \, dy \right)^\frac{1}{p' + p' \epsilon} & = \left(\frac{1}{|I|}  \int_I  \| W^{-\frac{1}{p}} (y) V_I  \|^{p' + p' \epsilon}  \, dy \right)^\frac{1}{p' + p' \epsilon} \\ & \lesssim \left(\frac{1}{|I|}  \int_I  \| W^{-\frac{1}{p}} (y) V_I  \|^{p' }  \, dy \right)^\frac{1}{p' } \\ & \lesssim \|W\|_{\text{A}_p} ^{\frac{1}{p}}. \end{align*}

Therefore, we have \begin{align*} m_I |V_I W^{-\frac{1}{p}} \vec{f} |  &  =  \frac{1}{|I|}  \int_I  |V_I     W^{-\frac{1}{p} } (y) \vec{f}(y)| \, dy
\\ & \leq    \left(\frac{1}{|I|}  \int_I  \|V_I W^{-\frac{1}{p}} (y) \|^{p' + p' \epsilon}  \, dy \right)^\frac{1}{p' + p' \epsilon} \left(\frac{1}{|I|}  \int_I  |\vec{f}(y)|^\frac{p + p\epsilon}{1 + p\epsilon} \, dy \right)^\frac{1 + p\epsilon}{p + p\epsilon}
 \\ & \lesssim \|W\|_{\text{A}_p} ^\frac{1}{p}  \left(\frac{1}{|I|}  \int_I  |\vec{f}(y)|^\frac{p + p\epsilon}{1 + p\epsilon} \, dy \right)^\frac{1 + p\epsilon}{p + p\epsilon} \nonumber  \end{align*}

 Which means that \begin{equation*} M_W ' \vec{f} (x) \leq \|W\|_{\text{A}_p} ^\frac{1}{p} (M |\vec{f}|^\frac{p + p\epsilon}{1 + p\epsilon} (x))^\frac{1 + p\epsilon}{p + p\epsilon} \end{equation*} (where $M$ is the ordinary unweighted maximal function) so that by the standard $L^{1 + \delta}$ bounds for the ordinary maximal function when $\delta > 0$ is small, we get

 \begin{equation*} \| M_W ' \|_{L^p \rightarrow L^p} \lesssim \epsilon^\frac{1}{p} \|W\|_{\text{A}_p} ^\frac{1}{p}  \lesssim \|W\|_{\text{A}_p} ^{\frac{1}{p} + \frac{p'}{p^2}} = \|W\|_{\text{A}_p} ^\frac{1}{p - 1}  \end{equation*} which completes the proof that (b) $\Rightarrow$ (a).

(a) $\Rightarrow$ (b): Fixing $J \in \D$, plugging in the test functions $\vec{f} := \chi_J \vec{e}_i$ into $\Pi_A$ for any orthonormal basis $\{\vec{e}_i\}_{i = 1}^n $ of $\C$, and using $(a)$ combined with dyadic Littlewood-Paley theory and the equivalence of the matrix norm and the Hilbert-Schmidt norm gives us that \begin{align*} \|\Pi_A\|_{L^p \rightarrow L^p} ^p  |J|  & \gtrsim    \int_{\Rd} \left( \sum_{I \in \D} \sum_{\varepsilon \in \S} \frac{\|V_I A_I ^\varepsilon m_I (\chi_J W^{- \frac{1}{p}} )\| ^2}{|I|} \chi_I (x) \right)^\frac{p}{2} \, dx \nonumber \\ & \geq  \int_J \left( \sum_{\varepsilon \in \S} \sum_{I \in \D(J)} \frac{\|V_I A_I ^\varepsilon m_I ( W^{- \frac{1}{p}} )\| ^2}{|I|} \chi_I (x) \right)^\frac{p}{2} \, dx \end{align*} which in conjunction with Lemma \ref{RedOp-AveLem} says that \begin{align*}  \sup_{J \in \D} \frac{1}{|J|} \ & \int_J  \left( \sum_{\varepsilon \in \S} \sum_{I \in \D(J)} \frac{\|V_I A_I ^\varepsilon V_I ^{-1}  \| ^2}{|I|} \chi_I (x) \right)^\frac{p}{2} \, dx  \\ & \lesssim \sup_{J \in \D} \frac{1}{|J|} \ \int_J  \left( \sum_{\varepsilon \in \S} \sum_{I \in \D(J)} \frac{\|V_I A_I ^\varepsilon V_I ' \| ^2}{|I|} \chi_I (x) \right)^\frac{p}{2} \, dx  \\ & \lesssim  \|W\|_{\text{A}_p} ^n \sup_{J \in \D} \frac{1}{|J|} \ \int_J \left( \sum_{\varepsilon \in \S}  \sum_{I \in \D(J)} \frac{\|V_I A_I ^\varepsilon m_I ( W^{- \frac{1}{p}} )\| ^2}{|I|} \chi_I (x) \right)^\frac{p}{2} \, dx. \end{align*} Condition $(b)$ now follows immediately from $(ii) \Longleftrightarrow (i)$ of Theorem $3.1$ in \cite{NTV}, which (after a trivial relabeling) says that for any nonnegative sequence $\{a_I\}_{I \in \D}$ of real numbers we have that \begin{equation*} \sup_{J \in \D} \left(\frac{1}{|J|} \sum_{I \in \D(J)} a_I^2 \right)^\frac12 \approx \sup_{J \in \D} \left( \frac{1}{|J|} \int_J \left( \sum_{I \in \D(J)} \frac{ a_I ^2}{|I|} \chi_I(x) \right)^\frac{p}{2} \, dx \right)^\frac{1}{p}. \end{equation*}

\noindent We now prove that (c) $\Rightarrow$ (b) and  (a) $\Rightarrow$ (c) for the case $2 \leq p < \infty$.

(c) $\Rightarrow$ (b) when $2 \leq p < \infty$:  Note that condition (c) is equivalent to \begin{equation*} \frac{1}{|K|}  \sum_{\varepsilon \in \S} \sum_{I \in \D(K)} \|V_K ^{-1} (A_I ^\varepsilon) ^* V_I ^2 A_I ^\varepsilon V_K ^{-1}  \|    \lesssim 1    \label{CondC'} \end{equation*} for any $K \in \D$.  Fix $J \in \D$ and for each $j \in \mathbb{N}$ let $\J^j(J)$ and $\F^j (J)$ be defined as they previously where $\lambda_1 \approx 1$, and $ \lambda_2 \approx \|W\|_{A_p} ^\frac{p'}{p}$ are picked so that Lemma \ref{DSTLem} is true.  Then  (\ref{STDef}) tells us that \begin{align*} \frac{1}{|J|} & \sum_{\varepsilon \in \S} \sum_{I \in \D(J)}   \|V_I A_I ^\varepsilon V_I ^{-1} \|^2
\\ & \leq  \frac{1}{|J|} \sum_{j = 1}^\infty  \sum_{\varepsilon \in \S} \sum_{K \in \J^{j - 1}(J) } \sum_{I \in \F(K)} \|V_I ^{-1} V_K\| \|V_K^{-1} (A_I ^\varepsilon) ^*  V_I ^2 A_I ^\varepsilon V_K ^{-1} \| \|V_K V_I ^{-1}\|  \\ & \lesssim  \|W\|_{\text{A}_p} ^{\frac{2}{p}} \frac{1}{|J|} \sum_{j = 1}^\infty \sum_{\varepsilon \in \S} \sum_{K \in \J^{j - 1}(J) } \sum_{I \in \D(K)}  \|V_K^{-1} (A_I ^\varepsilon) ^*  V_I ^2 A_I ^\varepsilon  V_K ^{-1} \| \\ & \lesssim  \|W\|_{\text{A}_p} ^{\frac{2}{p}} \frac{1}{|J|} \sum_{j = 1}^\infty \sum_{K \in \J^{j - 1}(J) } |K| \\ & \lesssim  \|W\|_{\text{A}_p} ^{\frac{2}{p}}  \sum_{j = 1}^\infty 2^{-j}  \lesssim \|W\|_{\text{A}_p} ^{\frac{2}{p}}. \end{align*}


 (a) $\Rightarrow$ (c) when $2 \leq p < \infty$: Fix $J \in \D$ and $\vec{e} \in \C$.  If $\vec{f} = W^\frac{1}{p} \chi_J \vec{e}$, then condition (a), the definition of $V_J$, and H\"{o}lder's inequality give us that \begin{align*} |J| |V_J \vec{e}|^p \|\Pi_A \|_{L^p \rightarrow L^p} ^p & \gtrsim \int_{\Rd} \left( \sum_{I \in \D} \sum_{\varepsilon \in \S} \frac{|V_I A_I ^\varepsilon m_I (\chi_J \vec{e}) |^2}{|I| } \chi_I (t) \right)^\frac{p}{2} \, dt \\ & \geq |J| \left[\frac{1}{|J|} \int_J \left( \sum_{\varepsilon \in \S} \sum_{I \in \D(J)} \frac{|V_I A_I ^\varepsilon  \vec{e}|^2 }{|I| } \chi_I (t) \right)^\frac{p}{2} \, dt \right] \\ & \geq |J| \left[ \frac{1}{|J|} \sum_{\varepsilon \in \S} \sum_{I \in \D(J)} |V_I A_I ^\varepsilon \vec{e}|^2 \right]^\frac{p}{2} \end{align*} which proves (c), and in fact shows that (a) $\Leftrightarrow$ (b)   $\Leftrightarrow$ (c) when $ 2 \leq p < \infty$.  We will now complete the proof when $1 < p \leq 2$.

 (b) $\Rightarrow$ (c) when $1 < p \leq  2$:  To avoid confusion in the subsequent arguments, we will write $V_I = V_I(W, p)$ to indicate which $W$ and $p$ the $V_I$ at hand is referring to.  As mentioned before, it is easy to see that $W$ is a matrix A${}_p$ weight if and only if $W^{1 - p'}$ is a matrix A${}_{p'}$ weight and that we may take $V_I(W^{1 - p'}, p') = V_I ' (W, p)$ and $V_I ' (W^{1 - p'}, p') =  V_I(W, p)$.  Now if (b) is true, then the two equalities above give us that \begin{align*} & \sup_{J \in \D} \frac{1}{|J|} \sum_{\varepsilon \in \S} \sum_{I \in \D(J)} \|V_I(W^{1 - p'}, p')( A_I ^\varepsilon) ^* V_I ' (W^{1 - p'}, p')\|^2 \nonumber \\ &=  \sup_{J \in \D} \frac{1}{|J|} \sum_{\varepsilon \in \S} \sum_{I \in \D(J)} \|V_I(W, p) A_I ^\varepsilon  V_I ' (W, p)\|^2   < \infty.  \end{align*} However, repeating word for word the proofs of (b) $\Rightarrow$ (a) $\Rightarrow$ (c) for the case $2 \leq p < \infty$ (where $W^{1 - p'}$ replaces $W$ and $A^* := \{(A_I ^\varepsilon) ^* : I  \in \D, \ \varepsilon \in \in \S\}$ replaces the sequence $A$) gives us that there exists $C > 0$ where \begin{equation*} \frac{1}{|J|} \sum_{\varepsilon \in \S}  \sum_{I \in \D(J)} A_I  ^\varepsilon (V_I(W^{1 - p'}, p'))^2 (A_I ^\varepsilon) ^*  < C (V_J(W^{1 - p'}, p'))^2,  \end{equation*} which proves (c) when $1 < p \leq  2$.  		\\

(c) $\Rightarrow$ (b) when $1 < p \leq  2$: This follows immediately by again replacing $W$ with  $W^{1 - p'}$, replacing $A$ with $A^* := \{(A_I ^\varepsilon) ^*\}_{I \in \D, \varepsilon \in \S}$, and using the proof of (c) $\Rightarrow$ (b) when $2 \leq p < \infty$. Since (a) $\Leftrightarrow$ (b) was shown for all $1 < p < \infty$, we therefore have that (a) $\Leftrightarrow$ (b) $\Leftrightarrow$  (c) for all $1 < p < \infty.$  The proof is now completed. \hfill $\square$

We can now prove \ref{ParaThm}  \\

\noindent \textit{Proof of Theorem} \ref{ParaThm}.  Note that  \eqref{LpEmbedding}, standard dyadic Littlewood-Paley theory, and the definition of $\pi_B$ gives us that  \begin{equation*} \|\pi_B  W^{-\frac{1}{p}} \vec{f} \|_{L^p(W)} ^p  \approx  \inrd \left(\sum_{\varepsilon \in \S} \sum_{I \in \D} \frac{| V_I B_I ^\varepsilon m_I (W^{-\frac{1}{p}} \vec{f})  |^2}{|I|} \chi_I(x) \right)^\frac{p}{2} \, dx \approx \| \Pi_B \vec{f} \|_{L^p} ^p \end{equation*}  where $\Pi_B $ is the operator in $(a)$ of \ref{CarEmbedThm} with respect to the Haar coefficient sequence $\{B_I ^\varepsilon \}_{I \in \D, \ \varepsilon \in \S}$. Thus, the $L^p(W)$ boundedness of $\pi_B$ is equivalent to the $L^p$ boundedness of $\Pi_B$.  Thanks to Theorem \ref{CarEmbedThm}, the proof will be completed by showing that $B \in \BMOW$ if and only if $\Pi_B$ is bounded on $L^p$.

 To that end, if again $\{\V{e}_i\}_{i = 1}^n$ is any orthonormal basis of $\C$, then  $\Pi_B$ being bounded on $L^p$ in conjunction with \eqref{LpEmbedding} gives us that \begin{align*}  \sup_{J \in \D} \ \frac{1}{|J|} &  \int_J \|W^\frac{1}{p} (x) (B(x) - m_J B) V_J ^{-1}  \|^p \, dx \\ & \approx   \sum_{i = 1}^n \sup_{J \in \D}  \ \frac{1}{|J|} \int_{\Rd} |W^\frac{1}{p} (x) \chi_J (x) (B(x) - m_J B) V_J ^{-1} \V{e}_i |^p \, dx \\ & \approx \sum_{i = 1}^n \sup_{J \in \D} \frac{1}{|J|}   \ \int_{\Rd} \left(\sum_{\varepsilon \in \S} \sum_{I \in \D(J)} \frac{|V_I B_I ^\varepsilon V_J ^{-1} \V{e}_i |^2}{|I|} \chi_I (x) \right)^\frac{p}{2} \, dx \\ & \leq \sum_{i = 1}^n \sup_{J \in \D}   \frac{1}{|J|} \int_{\Rd} \left(\sum_{I \in \D(J)} \sum_{\epsilon \in \S} \frac{|V_I B_I ^\varepsilon m_I (W^{-\frac{1}{p}} \{ \chi_J  W^\frac{1}{p}V_J ^{-1} \V{e}_i\}) |^2}{|I|} \chi_I (x) \right)^\frac{p}{2} \, dx  \\ & \lesssim \sum_{i = 1}^n \sup_{J \in \D} |J|^{-1}  \|\Pi_B \chi_J W^\frac{1}{p}  V_J ^{-1} \V{e}_i\|_{L^p} ^p \\ & \lesssim \|\Pi_B\|_{L^p} ^p
\end{align*} \noindent by the definition of $V_J,$ which means that the first condition of the definition of $\BMOW$ is true for all $1 < p < \infty$.   Similarly, the validity of the second condition of the definition of $\BMOW$ for all $1 < p < \infty$ if $\Pi_B$ is bounded follows by the above arguments in conjunction with the arguments used to prove $(b) \Rightarrow (c)$ when $1 < p \leq 2$ (that is, taking adjoints in condition $(b)$ and using the fact that $W$ is a matrix A${}_p$ weight if and only if $W^{1 - p'}$ is a matrix A${}_{p'}$ weight.

Now if $2 \leq p < \infty$ and $B \in \BMOW$ then as before \eqref{LpEmbedding} gives us that for any $\vec{e} \in \C$
\begin{align*}  \sup_{J \in \D} \ \frac{1}{|J|} &  \int_J |W^\frac{1}{p} (x) (B(x) - m_J B) V_J ^{-1} \vec{e}   |^p \, dx \\ & \approx    \sup_{J \in \D}  \ \frac{1}{|J|} \int_J \left(\sum_{\varepsilon \in \S} \sum_{I \in \D(J)} \frac{|V_I B_I ^\varepsilon V_J ^{-1} \vec{e}|^2}{|I|} \chi_I (x) \, dx \right) ^\frac{p}{2} \, dx   \\ & \geq  \sup_{J \in \D}  \ \left( \frac{1}{|J|}  \sum_{\varepsilon \in \S} \sum_{I \in \D(J)} |V_I B_I ^\varepsilon V_J ^{-1} \vec{e}|^2 \right)^\frac{p}{2} \end{align*} since $\frac{p}{2} \geq 1$, which says that condition $(c)$ (when $2 \leq p < \infty$) is true if $B \in \BMOW$.  The same argument using \eqref{LpEmbedding}  for the space $L^{p'}(W^{1 - p'})$  also shows that  condition $(c)$ (when $1 < p \leq 2$) is true if $B \in \BMOW$  since in this case $p' \geq 2$.    The proof is now completed.
\hfill $\Box$ \\

We now prove \eqref{ParaQuant} \\

\textit{Proof of} \eqref{ParaQuant}. By  \eqref{LowerBound} and the proof of (b) $\Rightarrow$ (a) we have \begin{align*} \|\pi_B  W^{-\frac{1}{2}} \vec{f} \|_{L^2 (W)} \lesssim \Atwo{W} ^\frac12  (\log \Atwo{W})^\frac12  \| \Pi_B \vec{f} \|_{L^2}  \lesssim \Atwo{W} ^\frac32  (\log \Atwo{W})^\frac12 \|B\|_{*} ^\frac12. \end{align*} \hfill $\square$.

We will end this section with an important comment. Note that the proof of (b)  $\Rightarrow$ (c) when $1 < p \leq 2$ immediately says that $B \in \BMOW$ if and only if $B^* \in \BMOWq$.  Thus, we have $B \in \BMOW$ if and only if $\pi_{B^*}$ is bounded on $L^p(W^{1 - p'})$ if and only if $(\pi_{B^*})^*$ is bounded on $L^p(W)$. We will record this as a corollary since we will need this fact when we prove sufficiency in the proof of Theorem \ref{CommutatorThm}.

\begin{corollary} \label{CorForSufficiency} If $W$ is a matrix A${}_p$ weight then $B \in \BMOW$ if and only if $(\pi_{B^*})^*$ is bounded on $L^p(W)$. \end{corollary}

\section{Proof of Theorem \ref{CommutatorThm}}

\subsection{Preliminaries} \label{Section31}

Before we prove Theorem \ref{CommutatorThm} we will need some preliminary results, including  Proposition \ref{HaarMultThm}. \\

\noindent \textit{Proof of Proposition} \ref{HaarMultThm}. If $M = \sup_{I \in \D, \varepsilon \in \S} \|V_I A_I ^\varepsilon V_I '\| < \infty$, then two applications of $(\ref{LpEmbedding})$ and \eqref{SmallBig} give us that \begin{align*} \|T_A \vec{f} \|_{L^p(W)} ^p & \approx \inrd \left(\sum_{I \in \D} \sum_{\varepsilon \in \S} \frac{|V_I A_I ^\varepsilon {\vec{f}}_I ^\varepsilon |^2}{|I|} \chi_I (x) \right)^\frac{p}{2} \, dx \\ & \leq \inrd \left(\sum_{I \in \D} \sum_{\varepsilon \in \S} \frac{\|V_I A_I ^\varepsilon V_I ^{-1} \|^2 | V_I {\vec{f}}_I ^\varepsilon|^2}{|I|} \chi_I (x) \right)^\frac{p}{2} \, dx \\ & \leq M^p \inrd \left(\sum_{I \in \D} \sum_{\varepsilon \in \S} \frac{ | V_I {\vec{f}}_I ^\varepsilon |^2}{|I|} \chi_I (x) \right)^\frac{p}{2} \, dx \\ &  \approx M^p \|\vec{f}\|_{L^p(W)} ^p. \end{align*}

For the other direction, let $\ell(J)$ denote the side-length of the cube $J$.  Fix some $J_0 \in \D$  and $\varepsilon' \in \S$, and let $J_0 ' \in \D(J_0)$ with $\ell(J_0 ') = \frac{1}{2} \ell(J_0)$. Again by  \eqref{LpEmbedding} we have that \begin{equation} \label{NecHaarMultEq}  \int_{\Rd} \left(\sum_{I \in \D} \sum_{\varepsilon \in \S} \frac{|V_I A_I ^\varepsilon (W^{-\frac{1}{p}} \vec{f} )_I ^\varepsilon |^2 }{|I|} \chi_I (x) \right) ^\frac{p}{2} \, dx \lesssim \|\vec{f}\|_{L^p}^p.   \end{equation}  Plugging $\vec{f} := \chi_{J_0 '} \vec{e}$ for any $\vec{e} \in \C$ into (\ref{NecHaarMultEq}) and noticing that \begin{equation*}  (W^{-\frac{1}{p}} \chi_{J_0 '} \vec{e} )_{J_0} ^{\varepsilon'} = \pm 2^{-\frac{d}{2}} |J_0| ^\frac{1}{2} m_{J_0 '} (W^{- \frac{1}{p}}\V{e}) \end{equation*} gives us (in conjunction with Lemma \ref{RedOp-AveLem}) that
\begin{equation*} \|V_{J_0} A_{J_0} ^{\varepsilon'} V_{J_0 '} '  \|  \lesssim \|W\|_{\text{A}_p} ^\frac{n}{p} \|V_{J_0} A_{J_0} ^{\varepsilon'} m_{J_0 '} (W^{- \frac{1}{p}}) \| \lesssim  \|W\|_{\text{A}_p} ^\frac{n}{p} \|T_A \|_{L^p \rightarrow L^p}.\end{equation*}   Using the definition of $V_{J_0 '} ' $ and summing over all of the $2^d$ first generation children $J_0 '$ of $J_0$ in conjunction with \eqref{SmallBig} finally (after taking the supremum over $J_0 \in \D$) gives us that

\begin{equation*} \sup_{J \in \D, \ \varepsilon \in \S} \|V_{J} A_{J}^\varepsilon V_{J} ^{-1}   \| \lesssim  \sup_{J \in \D, \ \varepsilon \in \S} \|V_{J} A_{J} ^\varepsilon  V_{J} '\|    < \|W\|_{\text{A}_p} ^\frac{n}{p} \|T_A \|_{L^p \rightarrow L^p}\end{equation*}  as desired. \hfill $\square$

To prove both necessity in Theorem \ref{CommutatorThm} and Proposition \ref{BMOvsDyadicBMO} we will need the following lemma.  While the simple proof is very similar to the proof of Lemma $6.2$ in \cite{IM}, we will nevertheless provide the details.

\begin{lemma} \label{InfLem}  Let $B$ be a locally integrable $\Mn$ valued function on $\Rd$,  $W$ a matrix A${}_p$ weight, and $Q$ a cube.  Then  \begin{align*} \left(\frac{1}{|Q|} \int_Q  \right. & \left.|W^\frac{1}{p}(x) ( B(x) - m_Q B )V_Q ^{-1} | ^p \, dx \right)^\frac{1}{p} \\ & \leq (1 + \|W\|_{\text{A}_p} ^\frac{1}{p})   \inf_{A \in \Mn} \left(\frac{1}{|Q|} \int_Q |W^\frac{1}{p}(x)  (B(x) - A)V_Q ^{-1} | ^p \, dx \right)^\frac{1}{p}. \end{align*} \end{lemma}

 \begin{proof}
 By the triangle inequality, \begin{align} \left(\frac{1}{|Q|} \int_Q |W^\frac{1}{p}(x)  (B(x) - m_Q B ) V_Q ^{-1} | ^p \, dx \right)^\frac{1}{p} & \leq \left(\frac{1}{|Q|} \int_Q |W^\frac{1}{p}(x) ( B(x) -  A)V_Q ^{-1} | ^p \, dx \right)^\frac{1}{p} \nonumber \\ &  + \left(\frac{1}{|Q|} \int_Q |W^\frac{1}{p}(x) ( A - m_Q B)V_Q ^{-1} | ^p \, dx \right)^\frac{1}{p}. \label{stuff} \end{align}  However,  \begin{align*} |W^\frac{1}{p}(x) & ( A - m_Q B) V_Q ^{-1} |^p  \\ & = \left| \frac{1}{|Q|} \int_Q W^\frac{1}{p}(x) (B(y) - A ) V_Q ^{-1}  \, dy \right|^p  \\ & = \left| \frac{1}{|Q|} \int_Q (W^\frac{1}{p}(x) W^{- \frac{1}{p}}(y)) W^\frac{1}{p}(y) (B(y) - A) V_Q ^{-1}  \, dy \right|^p \\ & \leq \left( \frac{1}{|Q|} \int_Q \|W^\frac{1}{p}(x) W^{- \frac{1}{p}}(y)\|^{p'} \, dy \right)^\frac{p}{p'} \left(\frac{1}{|Q|} \int_Q|   W^\frac{1}{p}(y) (B(y) - A )V_Q ^{-1} | ^p  \, dy \right).\end{align*}

 Plugging this into \eqref{stuff} and using \eqref{MatrixApDef} completes the proof.

\end{proof}

The proof strategy for sufficiency in Theorem \ref{CommutatorThm} will follow the simple arguments in \cite{LPPW}, though of course more care must be taken in our situation due to noncommutativity.   As in \cite{LPPW} the starting point is the fact that any of the Riesz transforms are in the $L^2$ SOT convex hull of the so called ``first order Haar shifts" (or for short, ``Haar shifts") which are defined by \begin{equation*} Q_\sigma h_I ^\varepsilon = h_{\sigma(I)} ^{\sigma(\varepsilon)} \end{equation*}  and where (slightly abusing notation in the obvious way) $\sigma : \D \times \S \rightarrow \D \times \S$ satisfies $2\ell(\sigma(I)) =  \ell(I)$ and $\sigma(I) \subseteq I$ for each $I \in \D$ (see \cite{H} for the definition of general Haar shifts, which are used to study general CZOs).  Fixing $\sigma$ and letting $Q = Q_{\sigma}$,  it is then enough to get an $L^p(W)$ bound on each $[B, Q]$. Before we do this, however, we will need to prove that first of all $Q$ is bounded on $L^p(W)$ if $W$ is a matrix A${}_p$ weight.  Note that this was in fact done for $p = 2$ in \cite{CW}.

\begin{proposition} \label{HaarBdd} Each of the Haar shifts $Q$ are bounded on $L^p(W)$ if $W$ is a matrix A${}_p$ weight. \end{proposition}

\begin{proof} The proof follows easily from two applications of \eqref{LpEmbedding}.  In particular, note that \begin{equation*} Q\vec{f} = \sum_{\varepsilon' \in \S} \sum_{I' \in \D} \vec{f}_{I'} ^{\varepsilon '} h_{\sigma(I')}  ^{\sigma(\varepsilon')}. \end{equation*}  If $\tilde{I}$ denotes the parent of $I \in \D$ then  \eqref{LpEmbedding} in conjunction with \eqref{SmallBig} gives us that \begin{align*} \|Q\vec{f}\|_{L^p(W)} ^p  & \lesssim  \inrd \left(\sum_{I \in \sigma(\D)} \sum_{ \varepsilon' \in \S} \frac{|V_I \vec{f} _{\tilde{I}} ^{\varepsilon'} |}{|I|} \chi_I (x) \right)^\frac{p}{2} \, dx \\ & \lesssim \inrd \left(\sum_{I \in \sigma(\D)} \sum_{ \varepsilon' \in \S} \frac{|V_{\tilde{I}} \vec{f} _{\tilde{I}} ^{\varepsilon'} |}{|{\tilde{I}}|} \chi_{\tilde{I}} (x) \right)^\frac{p}{2} \, dx \\ & \lesssim \|\vec{f}\|_{L^p(W)} ^p. \end{align*} \end{proof}

\subsection{Proof of sufficiency}
We now prove sufficiency in Theorem \ref{CommutatorThm} \\

\textit{Proof of Sufficiency}. First write \begin{equation*} B = \sum_{I' \in \D} \sum_{\varepsilon' \in \S} B_{I'} ^{\varepsilon'} h_{I'} ^{\varepsilon '}, \ \ \ \vec{f} = \sum_{I \in \D} \sum_{\varepsilon \in \S} \vec{f}_{I} ^{\varepsilon} h_{I} ^{\varepsilon }\end{equation*} so that \begin{align*} [B, Q] \vec{f} & = \sum_{I \in \D} \sum_{\varepsilon \in \S}  \left( B \vec{f}_I ^\varepsilon Q h_I ^\varepsilon - Q (B h_I ^\varepsilon) \vec{f}_I ^\varepsilon \right) \\ & = \sum_{I, I' \in \D} \sum_{\varepsilon, \varepsilon' \in \S}  \left(  B_{I'} ^{\varepsilon'}   h_{I'} ^{\varepsilon '} (Q h_I ^\varepsilon) \vec{f}_I ^\varepsilon - B_{I'} ^{\varepsilon '} (Q  h_{I'} ^{\varepsilon '}  h_I ^\varepsilon) \vec{f}_I ^\varepsilon \right)  \\ & = \sum_{I, I' \in \D} \sum_{\varepsilon, \varepsilon' \in \S}  B_{I'} ^{\varepsilon'} \left([h_{I'} ^{\varepsilon '}, Q]h_I ^\varepsilon \right) \vec{f}_{I} ^\varepsilon. \end{align*}

\noindent Clearly there is no contribution if $I \cap I' = \emptyset$ and otherwise we have that

\begin{equation} \label{e.cases}
[  {h _{I'} ^{\varepsilon '}},  Q]h _{I} ^{\varepsilon}
=
\begin{cases}
0  &   I\subsetneq I'
\\
\pm\abs{I} ^{-1/2}  h ^{\sigma (\varepsilon)} _{\sigma (I)}
- Q (h ^{\epsilon'} _{I} h ^{\epsilon} _{I} )
& I=I'
\\
h ^{\varepsilon'} _{\sigma (I)}h ^{\sigma(\varepsilon)} _{\sigma(I)}
\pm \abs{I} ^{-1/2} h ^{\sigma (\varepsilon ')} _{\sigma ^2(I)}
&  I'=\sigma (I)
\\
  h ^{ \varepsilon '} _{I'} Q (h_I ^\varepsilon  ) - Q (h_I ^\varepsilon  h ^{ \varepsilon '} _{I'})
& I'\subsetneq I  \textup{ and } I' \neq \sigma(I).

\end{cases}
\end{equation}

Note that we can disregard sign changes thanks to the unconditionality of Theorem \ref{CarEmbedThm},  \eqref{LpEmbedding}, and Proposition \ref{HaarMultThm}, and we will not comment on this further in the proof.  When $I=I'$ we need to bound the two sums
\begin{equation}\sum_{I \in \D}\sum_{\varepsilon, \varepsilon' \in \S }  B_{I} ^{\varepsilon'}\vec{f}_{I} ^\varepsilon |I|^{-1/2} h_{\sigma(I)}^{\sigma(\varepsilon)}  \;\textup{ and }\; Q\left(\sum_{I \in \D}\sum_{\varepsilon, \varepsilon' \in \S }  B_{I} ^{\varepsilon'}  \vec{f}_{I} ^\varepsilon |I|^{-1/2}h_{I}^{\psi_{\varepsilon'}(\varepsilon)}\right). \label{DiagCommTerm}\end{equation} where $\psi_{\varepsilon'}(\varepsilon)$ is the signature defined by

\begin{equation*} h_I ^{\psi_{\varepsilon'} (\varepsilon)} =  |I|^\frac{1}{2} h_I^{\varepsilon} h_I^{\varepsilon' } \end{equation*} which is cancellative if and only if $\varepsilon \neq \varepsilon'$.

 However,  if $B \in \BMOW$ then condition $(b)$ in Theorem \ref{CarEmbedThm}  tells us that for $\epsilon, \epsilon'$ fixed and $\tilde{J}$ being the parent of $J \in \D$

 \begin{equation*} \sup_{J \in \sigma(\D)}  \| V_{J} (|\tilde{J}| ^{-\frac12} B_{\tilde{J}} ^{\epsilon'}) V_{\tilde{J}} ^{-1} \| \lesssim  \sup_{J \in \sigma(\D)}  \| V_{\tilde{J}} (|\tilde{J}| ^{-\frac12} B_{\tilde{J}} ^{\epsilon'}) V_{\tilde{J}} ^{-1} \|  < \infty \end{equation*} so that the first sum in \eqref{DiagCommTerm} can be estimated in a manner that is very similar to the proof of sufficiency in Proposition \ref{HaarMultThm} (that is, using \eqref{LpEmbedding} twice).

Note that the second sum of \eqref{DiagCommTerm} when $\varepsilon \neq \varepsilon'$ is also ``Haar multiplier like" and can be estimated in exactly the same way as the first sum in \eqref{DiagCommTerm}.    On the other hand, when $\epsilon = \epsilon'$ the second sum of \eqref{DiagCommTerm} becomes   \begin{equation*} Q \left(\sum_{I \in \D} \sum_{\varepsilon \in \S}  B_{I} ^{\varepsilon}  \vec{f}_{I} ^{\varepsilon} \frac{\chi_I}{|I|}  \right) = Q (\pi_{B^*} )^* \vec{f}. \end{equation*}

\noindent   However, by Corollary \ref{CorForSufficiency} we have that $B \in \BMOW$ if and only if $(\pi_{B^*})^*$ is bounded on $L^p(W)$.

We now look at the case when $I'=\sigma (I)$ which clearly gives us two sums corresponding to the two terms in \eqref{e.cases}.  For the first term, we obtain the sum
\begin{align*} \sum_{I  \in \D} \sum_{\varepsilon, \varepsilon' \in \S}  B_{\sigma(I)} ^{\varepsilon'} h ^{\varepsilon'} _{\sigma (I)}h ^{{\sigma(\varepsilon)}} _{\sigma(I)} \vec{f}_{I} ^\varepsilon & =  \sum_{I  \in \D} \sum_{\varepsilon \in \S}   B_{\sigma(I)} ^{\sigma(\varepsilon)}  \vec{f}_{I} ^\varepsilon \frac{\chi_{\sigma(I)}}{|\sigma(I)|}  \\ & + \sum_{I  \in \D} \sum_{\substack{\varepsilon, \varepsilon'  \in \S \\ \varepsilon' \neq \sigma(\varepsilon) } } |I|^{-\frac12} B_{\sigma(I)} ^{\varepsilon'} h_{\sigma(I)} ^ {\psi_{\varepsilon'}(\sigma(\varepsilon))} \vec{f}_{I} ^\varepsilon.\end{align*}

\noindent However, a simple computation gives us  \begin{equation*} \sum_{I  \in \D} \sum_{\varepsilon \in \S}   B_{\sigma(I)} ^{\sigma(\varepsilon)}  \vec{f}_{I} ^\varepsilon \frac{\chi_{\sigma(I)}}{|\sigma(I)|}    = (\pi_{B^*} )^* Q \vec{f} \end{equation*} which is bounded on $L^p(W)$.  Also, the second sum is again  ``Haar multiplier like" and can be estimated in easily in a manner that is similar to the proof of sufficiency for Proposition \ref{HaarMultThm}.

 Furthermore, for the second sum in the two terms when $I' = \sigma(I)$, we need to bound
\begin{equation*} \sum_{I \in \D} \sum_{\varepsilon, \varepsilon' \in \S}  B_{\sigma(I)} ^{\varepsilon'} \abs{I} ^{-1/2} h ^{\sigma (\varepsilon ')} _{\sigma ^2(I)}  \vec{f}_{I} ^\varepsilon \end{equation*} which yet again is  ``Haar multiplier like" and can be estimated in a manner that is similar to the proof of sufficiency for Proposition \ref{HaarMultThm}

To finally finish the proof of sufficiency we bound the triangular terms.  First, if $I \supsetneq I'$ then obviously $h_I ^\varepsilon$ is constant on $I'$.  Thus,  \begin{align*}  \sum_{I'\in \D} \sum_{I \supsetneq I'} \sum_{\varepsilon, \varepsilon' \in \S}  B_{I'} ^{\varepsilon'} Q (h_I ^\varepsilon  h ^{ \varepsilon '} _{I'})  \vec{f}_I ^\varepsilon & =  \sum_{I' \in \D} \sum_{ \varepsilon'  \in \S} B_{I'} ^{\varepsilon'} Q(h ^{\varepsilon '} _{I'}) \sum_{I \supsetneq I'} \sum_{\varepsilon \in \S}    \vec{f}_{I} ^\varepsilon h_I ^\varepsilon  \\ & = \sum_{I' \in \D}\sum_{ \varepsilon' \in \S} B_{I'} ^{\varepsilon'} Q h_{I'} ^{\varepsilon'} m_{I'} \vec{f} \\ & = Q \pi_B \vec{f}. \end{align*}

\noindent Now clearly $h ^{ \varepsilon '} _{I'} Q (h_I ^\varepsilon  ) = 0$ if $I' \cap \sigma(I) = \emptyset$. Furthermore, since  $I \supsetneq I'$  and $I' \neq \sigma(I)$, we must have $\sigma(I) \supsetneq I'$  so that  \begin{align*}  \sum_{I'\in \D} \sum_{I \supsetneq I'} \sum_{\varepsilon, \varepsilon' \in \S}  B_{I'} ^{\varepsilon'} h ^{ \varepsilon '} _{I'} Q (h_I ^\varepsilon  ) \vec{f}_{I} ^\varepsilon  & = \sum_{I'\in \D} \sum_{\varepsilon' \in \S}  B_{I'} ^{\varepsilon'} h ^{ \varepsilon '} _{I'} \sum_{I: \sigma(I) \supsetneq I'}  \sum_{ \varepsilon \in \S} h_{\sigma(I)}  ^{\sigma(\varepsilon)}  \vec{f}_{I} ^\varepsilon \\ & = \sum_{I' \in \D} \sum_{ \varepsilon' \in \S} B_{I'} ^{\varepsilon'}  h_{I'} ^{\varepsilon'} m_{I'} (Q\vec{f}) \\ & = \pi_B Q \vec{f} \end{align*}

\noindent which is obviously bounded on $L^p(W)$. The proof of sufficiency is now complete.  \hfill $\square$.

We now prove \eqref{CommQuant}. \\

\textit{Proof of} \eqref{CommQuant}. Note that the simple arguments used to prove sufficiency in Proposition \ref{HaarMultThm} combined with \eqref{UpperBound} and \eqref{LowerBound} shows that (as was noticed in \cite{BPW, CW}) \begin{equation*} \|T_A \|_{L^2(W) \rightarrow L^2(W)} \lesssim \Atwo{W}^\frac32 \log \Atwo{W} \left(\sup_{I \in \D, \ \varepsilon \in \S} \|V_I A_I ^\varepsilon V_I ^{-1}\| \right). \end{equation*}

Also a careful reading of the proof of sufficiency in Theorem \ref{CommutatorThm} reveals that \begin{align*}  \|[T, B] \|_{L^2(W) \rightarrow L^2(W)} & \lesssim \|Q\|_{L^2(W) \rightarrow L^2(W)} \max \{\|\pi_B\|_{L^2(W)\rightarrow L^2(W)} ,  \|\pi_{B^*}\|_{L^2(W^{-1} )\rightarrow L^2(W^{-1})}\} \\ &  +  \Atwo{W} ^\frac32 \log \Atwo{W} \|B\|_{*} ^\frac12\end{align*} where again $\|B\|_*$ is the canonical supremum in condition (b) of Theorem \ref{CarEmbedThm}  \hfill $\square$.

\subsection{Proof of necessity}

For the proof of necessity in Theorem \ref{CommutatorThm} we will use some simple ideas from \cite{J}.  We will in fact prove the following more general result for commutators with kernels considered in \cite{J}.

\begin{theorem} Let $K : \mathbb{R}^d \backslash \{0\} \rightarrow \mathbb{R}$  be not identically zero, be homogenous of degree $-d$, have mean zero over the unit sphere $\partial \mathbb{B}_d$, and satisfy $K \in C^{\infty} (\partial \mathbb{B}_d)$ (so in particular $K$ could be any of the Riesz kernels).  If $T$ is the (convolution) CZO associated to $K$, then we have that $[T, B]$ being bounded on $L^p(W)$ implies that $B \in \BMOW$. \label{KernelThm} \end{theorem}

\begin{proof} First note that it is enough to prove that $B$ satisfies the first condition in the definition of $\BMOW$ when $2 \leq p < \infty$ and $[T, B]$ is bounded on $L^p(W)$ since \begin{equation*} (W^\frac{1}{p} [T, B] W^{-\frac{1}{p}} )^* = - W^{-\frac{1}{p}} [T, B^*] W^\frac{1}{p} \end{equation*} which will allow us to immediately conclude that the second condition in the definition of $\BMOW$ is true when $1 \leq p < 2$ as $W^{1 - p'}$ is a matrix A${}_{p'}$ weight and $2 \leq p' < \infty$.  Now by assumption, there exists $z_0 \neq 0 $  and $\delta > 0$ where $\frac{1}{K(x)}$ is smooth on $|x - z_0| < \sqrt{d} \delta$, and thus can be expressed as an absolutely convergent Fourier series \begin{equation*} \frac{1}{K(x)} = \sum a_n e^{i v_n \cdot x} \end{equation*} \noindent for $|x - z_0| < \sqrt{d} \delta$ (where the exact nature of the vectors $v_n$ is irrelevant.)  Set $z_1 = \delta^{-1} z_0$.  Thus, if $|x - z_1| < \sqrt{d}$, then we have by homogeneity

\begin{equation*} \frac{1}{K(x)} = \frac{\delta^{-d}}{K(\delta x)} = \delta^{-d} \sum a_n e^{i v_n \cdot (\delta x)}. \end{equation*} \noindent Now for any cube $Q = Q(x_0, r)$ of side length $r$ and center $x_0$, let $y_0 = x_0 - rz_1$ and $Q' = Q(y_0, r)$ so that $x \in Q$ and $y \in Q'$ implies that \begin{equation*} \left|\frac{x - y}{r} - z_1\right|  \leq \left| \frac{x - x_0}{r} \right| + \left| \frac{y - y_0}{r} \right|\leq \sqrt{d}. \end{equation*}

Let \begin{equation*} S_Q (x) = \chi_Q (x) \frac{V_Q ^{-1} (B^*(x) - m_{Q'} B^* ) W^{\frac{1}{p}} (x)} {\|V_Q ^{-1} (B^*(x) - m_{Q'} B^* ) W^{\frac{1}{p}} (x)\|} \end{equation*}  so that \begin{align}\frac{1}{r^d} &  \left\|\int_\Rd  W^{\frac{1}{p}} (x) (B(x) -B(y)) V_Q ^{-1} \frac{r^d K(x - y)}{K(\frac{x-y}{r})}  S_Q (x)  \chi_{Q'} (y) \, dy  \right\| \label{CommEst1} \\ & = \chi_Q (x)  \frac{1}{r^d} \left\| \int_{Q'}  W^{\frac{1}{p}} (x) (B(x) -B(y)) V_Q ^{-1}  \frac{(W^{\frac{1}{p}} (x)  (B(x) - m_{Q'} B ) V_Q ^{-1}  )^* }{\|W^{\frac{1}{p}} (x)  (B(x) - m_{Q'} B ) V_Q ^{-1} \|}  \, dy \right\| \nonumber \\ & = \chi_Q (x) \left\|\frac{ W^{\frac{1}{p}} (x)  (B(x) - m_{Q'} B ) V_Q ^{-1} (W^{\frac{1}{p}} (x)  (B(x) - m_{Q'} B ) V_Q ^{-1}  )^*}{\| W^{\frac{1}{p}} (x)  (B(x) - m_{Q'} B ) V_Q ^{-1}\|} \right\| \nonumber \\ & = \chi_Q (x) \| W^{\frac{1}{p}} (x)  (B(x) - m_{Q'} B ) V_Q ^{-1}\|. \nonumber \end{align}

However, \begin{align*} \eqref{CommEst1} &  \leq \sum |a_n| \left\| W^\frac{1}{p}(x) \left(\inrd (B(x) - B(y)) K(x - y)  e^{- i \frac{\delta}{r} v_n \cdot y} V_Q ^{-1}  \chi_{Q'} (y) \, dy \right)S_Q (x)  e^{ i \frac{\delta}{r} v_n \cdot x} \right\|  \\ & = \sum |a_n| \left\| W^{\frac{1}{p}}(x) ([T, B] g_n) (x) f_n(x) \right\| \\ & \leq  \sum |a_n| \left\| W^{\frac{1}{p}}(x) [T, B] g_n (x)  \right\| \end{align*} where \begin{equation*} g_n(y) =   e^{- i \frac{\delta}{r} v_n \cdot y } V_Q ^{-1} \chi_{Q'} (y) ,\ \ \ \ f_n(x) = S_Q(x)  e^{ i \frac{\delta}{r} v_n \cdot x } \end{equation*}  and where the last inequality follows from the fact that $\|f_n(x)\| \leq 1$ for a.e. $x \in \Rd$.

But as $|x_0 - y_0 | = r \delta^{-1} z_0$, we can pick some $C > 1$ only depending on $K$ where  $\tilde{Q} = Q(x_0, C r) $ satisfies $Q \cup Q' \subseteq \tilde{Q}$.  Combining this with the previous estimates, we have from the absolute summability of the $a_n's $ and the boundedness of $[T, B]$ that \begin{align*} \left(\int_{Q} \| W^{\frac{1}{p}} (x)  (B(x) - m_{Q'} B ) V_Q ^{-1}\| ^p \, dx \right) ^{\frac{1}{p}} & \leq \sum |a_n| \| W^{\frac{1}{p}} [T, B] g_n  \|_{L^p}  \\ & \leq  \sum |a_n| \| W^{\frac{1}{p}} [T, B] W^{-\frac{1}{p}} (W^\frac{1}{p} g_n) \|_{L^p} \\ & \leq   \sup_n  \| W^{\frac{1}{p}} g_n\|_{L^p } \\ & \leq \|\chi_{Q'}  W^{\frac{1}{p}} V_Q ^{-1}\|_{L^p} \\ & \lesssim \|W\|_{\text{A}_p} ^\frac{1}{p} \|\chi_{\tilde{Q}}  W^{\frac{1}{p}} V_{\tilde{Q}} ^{-1}\|_{L^p} \\ &  \lesssim   |Q|^\frac{1}{p} \|W\|_{\text{A}_p} ^\frac{1}{p} \end{align*} \noindent where the second to last inequality is due to \eqref{BigSmall}.  The proof is now complete thanks to Lemma \ref{InfLem}.
\end{proof}

Lastly in this section we will give quick proofs of Corollary \ref{BMODefCor} and Proposition \ref{BMOvsDyadicBMO}. \\

  \textit{Proof of Corollary} \ref{BMODefCor}. If $B \in \BMOW$ then from the proof of Theorem \ref{KernelThm} we have \begin{equation*}    \sup_{\substack{I \subset \R^d \\  I \text{ is a cube}}} \frac{1}{|I|} \int_I \|W^\frac{1}{p} (x) (B(x) - m_I B) V_I ^{-1}  \|^p \, dx < \infty. \end{equation*} On the other hand, again since  \begin{equation*} (W^\frac{1}{p} [T, B] W^{-\frac{1}{p}} )^* = - W^{-\frac{1}{p}} [T, B^*] W^\frac{1}{p} \end{equation*} we can again use the proof of Theorem \ref{KernelThm}  to get that the dual condition\begin{equation*} \sup_{\substack{I \subset \R^d \\  I \text{ is a cube}}} \frac{1}{|I|} \int_I \|W^{-\frac{1}{p}} (x) (B^* (x) - m_I B^*) (V_I ')^{-1}  \|^{p'} \, dx < \infty \end{equation*} is true since $W^{1 - p'}$ is a matrix A${}_{p'}$ weight.  The proof is now complete. \hfill $\square$ \\

 \textit{Proof of Proposition } \ref{BMOvsDyadicBMO}.  It is well known (see \cite{LN}) that if  $\D^t = \{2^{-k} ([0, 1) ^d + m  + (-1)^k t)  : k \in \Z, m \in \Z^d\}$, then for any cube $I$ there exists $1\leq t\leq 2^d $ and $I_t\in \D^t$ such that $I\subset I_t$ and $\ell(I_t)\leq 6\ell(Q)$.  Thus, Lemma \ref{InfLem} gives us

 \begin{align*} \frac{1}{|I|} \int_I \|W^\frac{1}{p} (x) (B(x) - m_I B) V_I ^{-1}  \|^p \, dx & \lesssim \frac{1}{|I|} \int_I \|W^\frac{1}{p} (x) (B(x) - m_{I_t} B) V_I ^{-1}  \|^p \, dx \\ & \lesssim \frac{1}{|I_t|} \int_{I_t} \|W^\frac{1}{p} (x) (B(x) - m_{I_t} B) V_{I_t}  ^{-1}  \|^p \, dx \end{align*} which completes the proof.  \hfill $\square$.

\section{Counterexamples and other quantitative estimates}
\label{ParaprodCounterex}
In this last section we will produce the counterexamples mentioned in the introduction and additionally prove quantitative matrix weighted bounds for maximal functions and sparse operators.

\subsection{Counterexamples}

For the rest of this section let \begin{equation*} A := \left(\begin{array}{cc} 0 & 1 \\  1 & 0 \end{array} \right) \text{ and }  \ W :=  \left(\begin{array}{cc} |x|^\alpha & 0 \\  0 & |x|^{-\alpha} \end{array} \right) \end{equation*} for $x \in \mathbb{R}$, where $0 < \alpha < 1$  so that $W$ is trivially a matrix A${}_2$ weight on $\R$ since $W$ is diagonal.  Also in this section let $\D$ be the standard dyadic grid on $\R$.

\begin{proposition}  There exists a sequence $\{A_I \}_{I \in \D}$  where $T_A$ is not bounded on $L^2(W)$.   \end{proposition}

\begin{proof}
We will in fact prove that there exists a constant sequence $\{A_I \}_{I \in \D}$ with the above property, and in particular let $\{A_I\}_{I \in \D}$  be the constant sequence $A_I  = A$.  For $I_N = [0, 2^{-N})$ we have  \begin{align*} \lim_{N \rightarrow \infty} \|(m_{I_N} W)^\frac12 A (m_{I_N} W) ^{-\frac12} \| & \geq \lim_{N \rightarrow \infty} \|(m_{I_N} W)^\frac12 A (m_{I_N} W^{-1}) ^{\frac12}\| \\ & = \lim_{N \rightarrow \infty} \frac{2^{\alpha N}}{(1 - \alpha)^2} = \infty. \end{align*} An application of Proposition \ref{HaarMultThm} now says that $T_A$ is not bounded on $L^2(W)$. \end{proof}

We will now show that $B \in \text{BMO}$ and $W$ being a matrix A${}_2$ weight is not sufficient for $\pi_B$ with respect to $\D$ to be   bounded on $L^2(W)$.

\begin{proposition} There exists a matrix function $B $ with scalar BMO entries where $\pi_B$ is not bounded on $L^2(W)$, and consequently $[T, B]$ is not bounded on $L^2(W)$ where $T$ is any of the Riesz transforms.  \end{proposition}

\begin{proof} Let $B (x) := (\log |x|) A$ and let $J_N = [2^{-N - 1}, 2^{-N})$ for $N \in \mathbb{N}$.   Now assume that $N \in \mathbb{N}$ is in fact large enough where  \begin{equation*}  \frac{1}{|J_N|} \sum_{I \in \D(J_N)} |b_I|^2 > \frac{1}{2} \|b\|_{\text{BMO}}.  \end{equation*}    Let $\vec{f}_N := \chi_{J_N}  W^{- \frac{1}{2}}  \vec{e} $ where \begin{equation*} \vec{e}  := \left(\begin{array}{c} 1 \\0 \end{array} \right). \end{equation*}

By \eqref{LpEmbedding} we have \begin{align} \|\pi_B W^{- \frac{1}{2}}  \vec{f}_N \|_{L^2(W)} ^2 & \gtrsim  \sum_{I \in \D({J_N})} |(m_I W)^\frac12  B_I m_I(W^{-1}) A \vec{e} |^2 \nonumber \\ & =   \sum_{I \in \D({J_N})} |b_I| ^2 |(m_I W)^\frac12 A  m_I(W^{-1})  \vec{e} |^2 \nonumber  \\ & \gtrsim     \frac{\|b\|_{\text{BMO}}}{2} 2^{3\alpha  N}. \nonumber \end{align}  However, \begin{align} \|\vec{f}_N\|_{L^2} ^2 & = \int_{J_N} |W^{- \frac{1}{2}} (t)   \vec{e} |^2 \, dt \nonumber \\ & \approx  2^{\alpha  N} \nonumber  \end{align} which shows that $\pi_B $ can not be bounded on $L^2(\R;\Ctwo)$.
\end{proof}

\subsection{Maximal function and sparse operator bounds}

We will end this paper with some quantitative weighted norm inequalities that were mentioned earlier in the paper.  Now let \begin{equation*} M_W ' \V{f} (x) = \sup_{I \ni x} \frac{1}{|I|} \int_I |(m_I (W^{-1}) )^{-\frac{1}{2}} W^{-\frac{1}{2}} (y) \V{f} (y)| \, dy \end{equation*} and
\begin{equation*} M_W  \V{f} (x) = \sup_{I \ni x} \frac{1}{|I|} \int_I |W^\frac{1}{2} (x) W^{-\frac{1}{2}} (y) \V{f} (y)| \, dy \end{equation*} where the supremum is over dyadic cubes $I$ taken from some fixed dyadic grid. Note that the proofs of the next three results are slight modifications to the corresponding ones in \cite{CG} (which is where $M_W$ and a slight variation of $M_W '$ were first defined).

\begin{lemma} \label{IntMaxEst} $M_W '$ is bounded on $L^2$ and in particular \begin{equation*} \|M_W '  \|_{L^2 \rightarrow L^2}^2  \lesssim \Atwo{W}. \end{equation*} \end{lemma}

\begin{proof}    By the (scalar) reverse H\"{o}lder inequality, as before,  we can pick $\epsilon \approx \Atwo{W} ^{-1}$ where \begin{equation*} \left( \frac{1}{|I|} \int_I \|W^{-\frac12} (y) (m_I (W^{-1})) ^{-\frac{1}{2}} \|^{2 + \epsilon} \, dy \right)^\frac{1}{2 + \epsilon}  \lesssim \left( \frac{1}{|I|} \int_I \|W^{-\frac12} (y) (m_I (W^{-1})) ^{-\frac{1}{2}} \|^{2 } \, dy \right)^\frac{1}{2} \lesssim 1.   \end{equation*}

\noindent Thus by H\"{o}lder's inequality we have \begin{align*} M_W ' \V{f} (x) &\leq \sup_{I \ni x} \left( \frac{1}{|I|} \int_I \|W^{-\frac12} (y) (m_I (W^{-1})) ^{-\frac{1}{2}} \|^{2 + \epsilon} \, dy \right)^\frac{1}{2 + \epsilon}
\left(\frac{1}{|I|} \int_I |\V{f}(y)|^{\frac{2 + \epsilon}{1 + \epsilon}} \, dy \right)^\frac{1 + \epsilon}{2 + \epsilon} \\& \lesssim (M(|\V{f}|^{\frac{2 + \epsilon}{1 + \epsilon}}) (x) )^\frac{1 + \epsilon}{2 + \epsilon} \end{align*}  where $M$ is the standard maximal function with respect to cubes.

Finally, as before, the usual $L^{1 + \delta} \rightarrow L^{1 + \delta}$ maximal function bound given by the Marcinkewicz interpolation theorem gives us that \begin{equation*} \inrd |M_W ' \V{f} (x)|^2 \, dx \leq \inrd (M(|\V{f}|^{\frac{2 + \epsilon}{1 + \epsilon}}) (x))^\frac{2 + 2\epsilon}{2 + \epsilon} \, dx \lesssim \epsilon ^{-1} \|\V{f} \|_{L^2} ^2 \end{equation*} which completes the proof as $\epsilon ^{-1} \approx \Atwo{W}$.  \end{proof}

\begin{lemma} If $Q$ is a cube and \begin{equation*} N_Q (x) = \sup_{x \in R \subseteq Q} \|W^\frac{1}{2} (x) (m_R (W^{-1}))  ^{\frac12}\| \end{equation*} then \begin{equation*}
\int_Q (N_Q (x) )^2 \, dx \lesssim |Q| \Atwo{W}  \end{equation*} \end{lemma}

\begin{proof} We truncate $W$ as in \cite{BPW} p. 1733. More precisely, write \begin{equation*} W(x) = \sum_{j = 1}^n \lambda_j (x) P_{E_j (x)} \end{equation*} where the $\lambda_j(x)$'s are the eigenvalues of $W(x)$ with corresponding eigenspaces $E_j(x)$ and $P_{E_j(x)}$ is the orthogonal projection onto $E_j(x)$. Now for $n \in \N$,  let $E_1 ^n(x), E_2 ^n(x)$, and $E_3 ^n(x)$ be the span of the eigenspaces corresponding to the eigenvalues $\lambda_j(x) \leq n^{-1}, \ n^{-1} < \lambda_j(x) < n$, \ and $\lambda_j(x) \geq n$, respectively.  Finally, define the truncation $W_n$ as \begin{equation*} W_n(x) = n^{-1} P_{E_1 ^n(x)} + P_{E_2 ^n(x)} W  (x) P_{E_2 ^n(x)} + n P_{E_3 ^n(x)}. \end{equation*}  It is then easy to see that $W_n \rightarrow W$ and $W_n ^{-1} \rightarrow W^{-1} $ pointwise a.e.,  $\Atwo{W_n} \lesssim \Atwo{W}$ for each $n$, and $W_n, W_n^{-1} \leq n\text{Id}_{d \times d}$ (see \cite{BPW}).   If \begin{equation*} N_Q ^n  (x) = \sup_{x \in R \subseteq Q} \|W_n ^\frac{1}{2} (x) (m_R (W_n^{-1}))  ^{\frac12}\| \end{equation*} then $\|(m_R (W_n^{-1}))  ^{-\frac12} - (m_R (W^{-1}))  ^{-\frac12}\| \rightarrow 0 $ as $n \rightarrow \infty$ by the dominated convergence theorem since clearly \begin{equation*} \|W_n ^{-1} - W^{-1}\| \leq \|W_n ^{-1}\| + \| W^{-1}\| \leq 2 \max\{1, \|W^{-1}\|\}. \end{equation*}  Thus, we have that \begin{equation*} \int_Q (N_Q(x))^2 \, dx  \leq \int_Q (\liminf_{n \rightarrow \infty} N_Q ^n (x))^2 \, dx \leq \liminf_{n \rightarrow \infty}  \int_Q (N_Q ^n (x))^2 \, dx. \end{equation*}  Obviously $N_Q ^n (x) \leq n^2$ so trivially there exists $B_n$ such that \begin{equation*}
\int_Q (N^n _Q (x) )^2 \, dx \lesssim B_n |Q|.   \end{equation*}

Putting this all together, it is enough to show that $B \lesssim \Atwo{W}$ if we assume that \begin{equation*} \int_Q (N_Q (x) )^2 \, dx \lesssim B |Q| \end{equation*} (or in other words we show in fact that $B_n \lesssim \Atwo{W}$.)  To that end, let $\{R_j\}$ be maximal subcubes of $Q$ satisfying \begin{equation*} \|(m_Q (W^{-1}))^{-\frac12}  (m_{R_j} (W^{-1}))^\frac12\| > C \end{equation*} for some large $C$ independent of $W$ to be determined.\\

Note that if $x \in Q \backslash \cup_j R_j$ then for any dyadic cube $x \in R \subset Q$ we have \begin{align*} \|W^\frac{1}{2} (x) (m_R (W^{-1}))  ^{\frac12}\| & \leq \|W^\frac{1}{2} (x) (m_Q (W^{-1}))^\frac12 \| \| (m_Q (W^{-1}))^{-\frac12}     (m_R (W^{-1}))  ^{\frac12}\|  \\ & \leq C \|W^\frac{1}{2} (x) (m_Q (W^{-1}))^\frac12 \|. \end{align*} so that \begin{equation*} \int_{Q \backslash \cup_j R_j} (N_Q(x))^2 \, dx \leq C \int_Q  \|W^\frac{1}{2} (x) (m_Q (W^{-1}))^\frac12 \|^2 \, dx \leq C \Atwo{W} |Q|. \end{equation*}

On the other hand, \begin{align*} C^2   \sum_j |R_j| & \leq  \sum_j |R_j | \|(m_Q W^{-1})^{-\frac12} (m_{R_j} (W^{-1}))^\frac12\|^2   \\ & \lesssim
\sum_j  \int_{R_j} \|W^{-\frac12} (x) (m_Q W^{-1})^{-\frac12} \|^2 \, dx  \lesssim   |Q|.  \end{align*} Thus for $C$ large enough independent of $W$ we have  $\sum_j |R_j| \leq \frac12 |Q|$.

Clearly by the definition of $R_j$ and their maximality we can assume for each $x \in R_j$ that $N_Q(x) = N_{R_j} (x)$ since otherwise  $N_Q (x) \lesssim C \|W^\frac{1}{2} (x) (m_Q (W^{-1}))^\frac12 \|  $.  Thus without loss of generality \begin{equation*}  \int_{\cup_j R_j} (N_Q(x))^2 \, dx = \sum_j \int_{R_j} (N_{R_j} (x))^2 \, dx \leq B \sum_j |R_j| \leq \frac{1}{2} B |Q|. \end{equation*}  Finally this implies that there exists $C$ independent of $W$ where  $B \leq    \frac{1}{2}B + C \Atwo{W} $ which completes the proof.  \end{proof}

\begin{theorem} $M_W$ is bounded on $L^2$ and in fact \begin{equation*} \|M_W \|_{L^2 \rightarrow L^2} \lesssim \Atwo{W}.  \end{equation*}  \end{theorem}

\begin{proof} For each $x \in \Rd$ pick (and fix) some dyadic cube $R_x$ such that \begin{align} \frac12 M_W (\V{f})(x) &\leq \frac{1}{|R_x|} \int_{R_x}  |W^\frac12 (x) W^{-\frac12} (y) \V{f}(y)| \, dy  \label{RxDef} \\ & \leq \|W^\frac12 (x) (m_{R_x} (W^{-1}) )^{\frac{1}{2}} \| \left(\frac{1}{|R_x|} \int_{R_x} |(m_{R_x} (W^{-1}) )^{-\frac{1}{2}} W^{-\frac12} (y) \V{f}(y)| \, dy \right). \nonumber \end{align}

 For $x \in \Rd$ pick $j \in \Z$ where   \begin{equation} 2^j \leq  \int_{R_x} |(m_{R_x} (W^{-1}) )^{-\frac{1}{2}} W^{-\frac12} (y) \V{f}(y)| \, dy < 2^{j + 1} \label{SjDef} \end{equation} and let $S_j$  be the collection of all cubes $R = R_x$ for all $x \in \Rd$ that are maximal and satisfy (\ref{SjDef}) (note that the Cauchy Schwarz inequality implies that such a maximal cube exists).   We therefore have that for every $x \in \Rd$,  $x \in R_x \subseteq S$ for some $S \in S_j$ where $j = j_x \in \Z$. Note that if $R_x$ satisfies (\ref{SjDef}) for $j \in \Z$ then  \begin{equation*} \frac{1}{|R_x|} \int_{R_x} |(m_{R_x} (W^{-1}) )^{-\frac{1}{2}} W^{-\frac12} (y) \V{f}(y)| \, dy \leq  \frac{2}{|S|}\int_{S} |(m_{S} (W^{-1}) )^{-\frac{1}{2}} W^{-\frac12} (y) \V{f}(y)| \, dy \end{equation*} since otherwise trivially (\ref{SjDef}) is violated.

   Now if $x \in \Rd$ then pick $j = j_x $ as before and pick $S \in S_j$ with $R_x \subseteq S \in S_j$ so \begin{align*}  M_W (\V{f})(x)  & \leq 2 \|W^\frac12 (x) (m_{R_x} (W^{-1}) )^{\frac{1}{2}} \| \left(\frac{1}{|R_x|} \int_{R_x} |(m_{R_x} (W^{-1}) )^{-\frac{1}{2}} W^{-\frac12} (y) \V{f}(y)| \, dy \right) \nonumber  \\ & \leq 4 N_{S} (x)  \left(\frac{2}{|S|}\int_{S} |(m_{S} (W^{-1}) )^{-\frac{1}{2}} W^{-\frac12} (y) \V{f}(y)| \, dy \right) \\ & \leq 4 2 ^{j + 1} N_{S} (x)    \end{align*} so that finally the previous two lemmas give us that  \begin{align*} \inrd |M_W \V{f} (x) |^2 \, dx &\lesssim  \sum_{j \in \Z,\  S \in S_j}  2^{2j} \int_S (N_S (x)) ^2 \, dx \\ & \lesssim \Atwo{W} \sum_{j \in \Z} 2 ^{2j} |\bigsqcup_{S \in S_j} S|  \\ & \leq \Atwo{W} \sum_{j \in \Z} 2 ^{2j} |\{x : M_W ' \V{f} (x) > 2^j\}|  \\ & \approx   \Atwo{W} \|M_W ' \V{f} \|_{L^2} ^2  \\ & \lesssim \Atwo{W}^2 \|\V{f}\|_{L^2} ^2 \end{align*} which completes the proof.

\end{proof}

Interestingly, note that Lemma \ref{IntMaxEst} is sharp with respect to $\Atwo{W}$ since otherwise we could get a better $\Atwo{W}$ bound for $\|M_W\|_{L^2 \rightarrow L^2}$ (which is known to be sharp in the scalar setting, and thus the matrix setting). Despite this, the following simple result says that we can legitimately consider $M_W ' $ to be ``the" universal $p = 2$ matrix weighted maximal function corresponding to $W^{-1}$.

\begin{proposition} $M_W'$ is weak $(2, 2)$ for \textit{any} (not necessarily A${}_2$) matrix weight $W$. \end{proposition}

\begin{proof} Let $\lambda > 0$ and let  $\{I_j\}$ be the collection of maximal dyadic cubes such that \begin{equation*} \frac{1}{|I_j|} \int_{I_j} | (m_{I_j} W^{-1})^{-\frac12} W^{-\frac12} (y)  \V{f}(y)| \, dy > \lambda \end{equation*} so as usual $\{x : M_W ' \V{f} (x) > \lambda \} = \bigsqcup_j I_j.$    Then \begin{align*} \sum_j |I_j| & = \sum_{j} \frac{|I_j|^2}{|I_j|}  \leq \frac{1}{\lambda^2} \sum_j \frac{1}{|I_j|} \left(\int_{I_j} | (m_{I_j} W^{-1})^{-\frac12} W^{-\frac12} (y) \V{f}(y)| \, dy \right)^2 \\ & \leq \frac{1}{\lambda^2} \sum_j  \left( \frac{1}{|I_j|}\int_{I_j} \| (m_{I_j} W^{-1})^{-\frac12} W^{-\frac12} (y) \|^2 \, dy \right)\left(\int_{I_j} |\V{f}(y)|^2 \, dy \right) \\ & \lesssim \frac{1}{\lambda^2} \sum_j  \int_{I_j} |\V{f}(y)|^2 \, dy \leq \frac{\|\V{f}\|_{L^2}^2}{\lambda^2}. \end{align*}  \end{proof}

Note that very similar maximal functions can be defined when $p \neq 2$ and similar weighted norm inequalities can be proved for these maximal functions (see \cite{IM} for proofs, where in fact fractional matrix weighted maximal functions are studied in detail and are applied to the study of matrix weighted norm inequalities for fractional integral operators and related matrix weighted Poincare and Sobolev inequalities.)

Lastly we will give a very simple ``maximal function" proof of the matrix weighted norm inequalities from \cite{BW} for sparse operators, which provides a simpler proof that avoids the Carleson embedding theorem and that is similar to the by now classical proof from \cite{CMP}.   Also, note that in the scalar weighted case, the $L^2(W)$ bound of sparse operators (see the definition below) has linear $\Atwo{w}$ dependence (see \cite{CMP} for the very easy maximal function proof in the scalar setting).

Finally, despite it's relative ease, note that this proof clearly highlights one of the severe challenges in using maximal functions (in any way, shape, or form) to prove sharp matrix weighted norm inequalities:  the absence of $L^2$ bounds independent of $W$ for the  universal matrix weighted maximal function makes scalar arguments much less efficient in the matrix weighted setting.

Let $\MC{G} \subset \D$ be a ``sparse" collection in the sense that for any $I \in \MC{G}$ we have \begin{equation*} \sum_{J \in \chG(I)} |J| \leq \frac12 |I| \end{equation*} where $\chG(I)$ are the children of $I$ that are also members of $\MC{G}$.  Furthermore, define the sparse operator $S = S_{\MC{G}}$ by \begin{equation*} S \vec{f} = \sum_{I \in \MC{G}} m_I \vec{f} \, \chi_I. \end{equation*}

\begin{proposition} If $W$ is a matrix A${}_2$ weight and $S$ is a sparse operator then \begin{equation*} \|S\|_{L^2(W) \rightarrow L^2(W)} \lesssim \Atwo{W} ^\frac32. \end{equation*}
\end{proposition}

\begin{proof} Let $\vec{f}, \vec{g} \in L^2$. For any $I \in \MC{G}$ let $E_I$ be defined by \begin{equation*} E_I = I \backslash \left(\bigcup_{J \in \chG(I)}  J \right) \end{equation*} so that clearly $\{E_I\}_{I \in \MC{G}}$ is a disjoint collection of measurable sets satisfying (by the sparseness condition) $2|E_I| \geq |I|$   .     Then \begin{align*} & \left|\ip{W^\frac12 S (W^{-\frac12} \vec{f} )}{\vec{g}}_{L^2} \right| \\  & =  \left|\ip{ S (W^{-\frac12}\vec{f}) }{W^\frac12 \vec{g}}_{L^2} \right| \\ & \leq \sum_{I \in \MC{G}} |I| \left|\ip{m_I (W^{-\frac12}\vec{f})}{m_I (W^\frac12 \vec{g})}_{\C}\right| \\ & \leq   \Atwo{W}^\frac12 \sum_{I \in \MC{G}} |I|  \left(m_I |(m_I (W ^{-\frac12}))^{-\frac12} W^{-\frac12}\vec{f}| \right)  \, \left(m_I | (m_I W ^{\frac12})^{-\frac12}  (W^\frac12 \vec{g})| \right)\\ & \leq 2 \Atwo{W}^\frac12 \sum_{I \in \MC{G}} |E_I|  \left(m_I |(m_I (W ^{-\frac12}))^{-\frac12} W^{-\frac12}\vec{f}| \right) \, \left(m_I | (m_I W ^{\frac12})^{-\frac12}  (W^\frac12 \vec{g})|  \right) \\ & \leq 2 \Atwo{W}^\frac12 \sum_{I \in \MC{G}} \int_{E_I} M_W ' \vec{f} (x) M_{W^{-1}} ' \vec{g} (x) \, dx \\ & \lesssim   \Atwo{W}^\frac32 \|\vec{f}\|_{L^2} \|\vec{g}\|_{L^2}. \end{align*} \end{proof}


\section{Acknowlegements} \label{ackref}The first author would like to thank Kelly Bickel and Brett Wick for their interesting discussions regarding sharp matrix weighted norm inequalities.

\end{document}